\documentclass[12pt,leqno]{amsart}

\usepackage{amsmath,amsfonts,amssymb,amsthm}
\usepackage[small,nohug,heads=vee]{diagrams}
\diagramstyle[labelstyle=\scriptstyle]

\theoremstyle{plain}
\newtheorem{theorem}{Theorem}
\newtheorem{corollary}[theorem]{Corollary}
\newtheorem{lemma}[theorem]{Lemma}
\newtheorem{proposition}[theorem]{Proposition}

\theoremstyle{definition}
\newtheorem{definition}[theorem]{Definition}
\newtheorem{example}[theorem]{Example}
\newtheorem{remark}[theorem]{Remark}

\newtheorem{question}[theorem]{Question}

\def\vi{\varphi}

\newcommand{\barint}{
\rule[.036in]{.12in}{.009in}\kern-.16in \displaystyle\int }

\newcommand{\barcal}{\mbox{$ \rule[.036in]{.11in}{.007in}\kern-.128in\int $}}

\newcommand{\partl}[2]{\frac{\partial #1}{\partial{#2}}}

\newcommand{\norm}[1]{\left\vert #1 \right \vert}
\newcommand{\Norm}[1]{\left\Vert #1 \right \Vert}

\newcommand{\mst}{\;:\;}

\newcommand{\B}{\mathbb B}
\newcommand{\R}{\mathbb R}
\newcommand{\C}{\mathbb C}
\newcommand{\Z}{\mathbb Z}

\newcommand{\Heis}{\mathbb H}
\newcommand{\Sph}{\mathbb S}

\newcommand{\lip}{{\rm Lip}\, }
\newcommand{\hor}{{\rm H}\, }

\newcommand{\cH}{\mathcal H}

\newcommand{\cU}{\mathcal U}
\newcommand{\cJ}{\mathcal J}
\newcommand{\cM}{\mathcal M}

\newcommand{\eps}{\epsilon}

\newcommand{\boldg}{{\mathbf g}}
\newcommand{\boldi}{{\mathbf i}}

\DeclareMathOperator{\diam}{diam}

\DeclareMathOperator{\dist}{dist}
\DeclareMathOperator{\Imag}{Im}
\DeclareMathOperator{\Real}{Re}

\DeclareMathOperator{\rank}{rank}

\makeatletter
\def\Ddots{\mathinner{\mkern1mu\raise\p@
\vbox{\kern7\p@\hbox{.}}\mkern2mu
\raise4\p@\hbox{.}\mkern2mu\raise7\p@\hbox{.}\mkern1mu}}
\makeatother

\numberwithin{theorem}{section} 
\numberwithin{equation}{section}

\title{On the lack of density of Lipschitz mappings in Sobolev spaces with Heisenberg
  target}

\author{Noel DeJarnette}
\address{Department of Mathematics, University of Illinois at
  Urbana-Champaign, 1409 West Green St., Urbana, IL 61801, {\tt ndejarne@illinois.edu}}

\author{Piotr Haj{\l}asz}
\address{Department of Mathematics, University of Pittsburgh, 301
  Thackeray Hall, Pittsburgh, PA 15260, USA, {\tt hajlasz@pitt.edu}}

\author{Anton Lukyanenko}
\address{Department of Mathematics, University of Illinois at
  Urbana-Champaign, 1409 West Green St., Urbana, IL 61801, {\tt lukyane2@illinois.edu}}

\author{Jeremy T. Tyson}
\address{Department of Mathematics, University of Illinois at
  Urbana-Champaign, 1409 West Green St., Urbana, IL 61801, {\tt tyson@math.uiuc.edu}}

\date{\today}

\thanks{
\\
{\bf Acknowledgements.} P.H. acknowledges support from NSF grant
DMS 0900871 ``Geometry and topology of weakly differentiable
mappings into Euclidean spaces, manifolds and metric spaces''.
J.T.T. acknowledges support from NSF grants DMS 0555869
``Nonsmooth methods in geometric function theory and geometric
measure theory on the Heisenberg group'' and DMS 0901620
``Geometric analysis in Carnot groups''. N.D. and A.L. acknowledge
support from NSF grant DMS 0838434 ``EMSW21-MCTP: Research
Experiences for Graduate Students''. N.D. acknowledges support
from NSF grants DMS 0901620 and DMS 0900871. N.D. also
acknowledges the Department of Mathematics at the University of
Pittsburgh for its hospitality during the academic year
2009--2010.}

\begin{document}

\sloppy


\begin{abstract}
We study the question: {\em when are Lipschitz mappings dense in the
  Sobolev space $W^{1,p}(M,\Heis^n)$?}
Here $M$ denotes a compact Riemannian manifold with or without
boundary, while $\Heis^n$ denotes the $n$th Heisenberg group equipped
with a sub-Riemannian metric. We show that Lipschitz maps are dense in
$W^{1,p}(M,\Heis^n)$ for all $1\le p<\infty$ if $\dim M \le n$, but
that Lipschitz maps are not dense in $W^{1,p}(M,\Heis^n)$ if $\dim M
\ge n+1$ and $n\le p<n+1$. The proofs rely on the construction of
smooth horizontal embeddings of the sphere $\Sph^n$ into $\Heis^n$. We
provide two such constructions, one arising from complex hyperbolic
geometry and the other arising from symplectic geometry. The
nondensity assertion can be interpreted as nontriviality of the $n$th
Lipschitz homotopy group of $\Heis^n$. We initiate a study of
Lipschitz homotopy groups for sub-Riemannian spaces.
\end{abstract}

\maketitle

\setcounter{tocdepth}{1}
\tableofcontents

\newpage

\section{Introduction}
\label{sec:intro}

In this paper we study Sobolev mappings from a compact Riemannian
manifold or from a domain in Euclidean space into the Heisenberg
group. The paper is motivated by recent developments in the theory of
Sobolev mappings into metric spaces and, in particular, by the work of
Capogna and Lin \cite{capognal} on harmonic mappings into the
Heisenberg group.

The main question which we investigate in this paper is the problem of
density of Lipschitz mappings. In a more classical setting the
question whether smooth mappings are dense in the space of Sobolev
mappings between manifolds was raised by Eells and Lemaire
\cite{eellsl}. In the case of manifolds, smooth mappings are dense
if and only if Lipschitz mappings are dense. For Heisenberg targets it
is more natural to ask about density of Lipschitz mappings.

Let $M$ and $N$ be compact Riemannian manifolds, $\partial
N=\emptyset$. We can always assume that $N$ is isometrically
embedded into a Euclidean space $\R^\nu$ (by the Nash theorem), and in
this case we define the class of Sobolev mappings $W^{1,p}(M,N)$,
$1\leq p<\infty$, as follows:
$$
W^{1,p}(M,N)=\{ u\in W^{1,p}(M,\R^\nu):\, u(x)\in N\ \mbox{a.e.}\}.
$$
The space $W^{1,p}(M,N)$ is equipped with a metric inherited from
the norm in $W^{1,p}(M,\R^\nu)$. Although every Sobolev mapping
$u\in W^{1,p}(M,N)$ can be approximated by smooth mappings with
values into $\R^\nu$, it is not always possible to find an
approximation by mappings from the space $C^\infty(M,N)$. A famous
example of Schoen and Uhlenbeck \cite{schoenu1}, \cite{schoenu2}
illustrates the issue. Consider the radial projection map
$u_0:\B^{n+1}\to\Sph^n$ from the closed unit ball in $\R^{n+1}$ onto its
boundary, given by
\begin{equation}
\label{cavitation} u_0(x)=\frac{x}{|x|}\, .
\end{equation}
The mapping $u_0$ has a singularity at the origin, but one can
easily prove that $u_0\in W^{1,p}(\B^{n+1},\Sph^n)$ for all $1\leq
p<n+1$. Schoen and Uhlenbeck proved that $u_0$ cannot be
approximated by $C^\infty(\B^{n+1},\Sph^n)$ mappings when $n\leq p<n+1$.

By using this construction one can prove that if $M$ is a compact
Riemannian manifold of dimension at least $n+1$, then
$C^\infty(M,\Sph^n)$ mappings are not dense in $W^{1,p}(M,\Sph^n)$
when $n\leq p<n+1$, see \cite{bethuelz}, \cite{hajlaszSobolev}. On
the other hand, Bethuel and Zheng \cite{bethuelz} proved that for
any manifold $M$, $C^\infty(M,\Sph^n)$ mappings are dense in
$W^{1,p}(M,\Sph^n)$ whenever $1\leq p<n$. Let us also mention that
Schoen and Uhlenbeck \cite{schoenu1}, \cite{schoenu2}, proved that
if $\infty>p\geq n=\dim M$, then $C^\infty(M,N)$ mappings are
dense in $W^{1,p}(M,N)$ for any target $N$.

Since in the case of manifold targets, density of smooth maps is equivalent to
density of Lipschitz maps, the above results imply the following well known
\begin{proposition}
\label{T0}
(a) If $M$ is a compact Riemannian manifold of dimension
$\dim M\geq n+1$, then Lipschitz mappings $\lip(M,\Sph^n)$ are not dense in
$W^{1,p}(M,\Sph^n)$ when $n\leq p<n+1$.

(b) If $M$ is a compact Riemannian manifold of dimension $\dim M\leq n$,
then Lipschitz mappings $\lip(M,\Sph^n)$ are dense in $W^{1,p}(M,\Sph^n)$
for all $1\leq p<\infty$.
\end{proposition}

In general, the answer to the question whether for given two
manifolds $M,N$ and $1\leq p<n=\dim M$, $C^\infty(M,N)$ mappings
are dense in $W^{1,p}(M,N)$ depends on the topological structure
of the manifolds. A complete solution to the problem has been
obtained by Hang and Lin \cite{hangl2} who corrected earlier
results of Bethuel \cite{bethuel1} and generalized results of
Haj\l{}asz \cite{hajlaszApproximation}. We refer the reader to the survey
\cite{hajlaszSobolev} for a detailed discussion and references.

There are several equivalent ways to define the space of Sobolev
mappings with target the Heisenberg group $\Heis^n$. We refer to
\cite{gromov-cc}, \cite{montgomery} or \cite{cdpt} for the basic
theory of the Heisenberg group and sub-Riemannian geometry; see
also section \ref{SR}. In this paper we define such mappings
using an isometric embedding of $\Heis^n$ into a Banach space. We
equip $\Heis^n$ with its standard Carnot-Carath\'eodory metric. As we
will see later on, the answer to the question of density of Lipschitz
maps in spaces of Sobolev maps with metric space targets can
potentially change under a bi-Lipschitz deformation of the target
space. Our results hold for a variety of metrics on $\Heis^n$ which
are bi-Lipschitz equivalent with the Carnot-Carath\'eodory metric.

Every separable metric space, in particular the Heisenberg group,
can be isometrically embedded into $\ell^\infty$ (e.g., via the Kuratowski
embedding \eqref{kuratowski}). Thus we may assume that
$\Heis^n\subset\ell^\infty$ and for a compact Riemannian manifold $M$
(with or without boundary) we define 
$$
W^{1,p}(M,\Heis^n)= \{ u\in W^{1,p}(M,\ell^\infty):\,
u(x)\in\Heis^n\ \mbox{a.e.}\}.
$$
The vector valued Sobolev space $W^{1,p}(M,\ell^\infty)$ is a
Banach space and $W^{1,p}(M,\Heis^n)$ is equipped with the metric
inherited from the norm. Our definition is different than that
provided by Capogna and Lin \cite{capognal}, but it is equivalent,
see Subsection~\ref{capogna-lin-remark} for details.

One may wonder if it would be possible to define the space
$W^{1,p}(M,\Heis^n)$ via an isometric embedding of $\Heis^n$ into a
Banach space different than $\ell^\infty$. Recently Cheeger and
Kleiner \cite[Theorem~1.6]{CK1}, \cite[Theorem~4.2]{CK2} proved that
$\Heis^n$ does not admit a bi-Lipschitz embedding into any Banach
space with the Radon-Nikodym property. Separable dual spaces and
reflexive spaces (separable or not) have the Radon-Nikodym property.
In view of this result there is little hope to replace $\ell^\infty$
by a better space in the definition of $W^{1,p}(M,\Heis^n)$.

As in the case of mappings between manifolds, Lipschitz mappings
$\lip(M,\ell^\infty)$ are dense in $W^{1,p}(M,\Heis^n)$, but the
question is whether we can take the approximating sequence from
$\lip(M,\Heis^n)$. It turns out that this is not always possible.

\

\begin{theorem}
\label{T1.5}
Equip $\Heis^n$ with the Carnot-Carath\'eodory metric.
\begin{itemize}
\item[(a)] If $M$ is a compact Riemannian manifold of dimension $\dim M\geq
  n+1$, then Lipschitz mappings $\lip(M,\Heis^n)$ are not dense in
  $W^{1,p}(M,\Heis^n)$ when $n\leq p<n+1$.
\item[(b)] If $M$ is a compact Riemannian manifold of dimension $\dim M\leq
  n$, then Lipschitz mappings $\lip(M,\Heis^n)$ are dense in
  $W^{1,p}(M,\Heis^n)$ for all $1\leq p<\infty$.
\end{itemize}
\end{theorem}

The proof of the failure of density in the case $n\leq p<n+1$ and $\dim M\geq
n+1$ is based on the following special case of Theorem~\ref{T1.5}(a).

\begin{proposition}\label{T1}
Lipschitz mappings $\lip(\B^{n+1},\Heis^n)$ are not dense in
$W^{1,p}(\B^{n+1},\Heis^n)$, when $n\leq p<n+1$.
\end{proposition}

Once this result is proved, the case of general $M$ follows from an elementary
surgery argument which allows us to build mappings from $M$ to $\Heis^n$ that
contain the mappings constructed in Proposition~\ref{T1} on
$(n+1)$-dimensional slices.

\begin{question}\label{Q1}
Are Lipschitz maps dense in $W^{1,p}(\B^{n+1},\Heis^n)$ when $p\geq n+1$
or $1\leq p<n$?
\end{question}

We have stated Theorem~\ref{T1.5} for the Carnot-Carath\'eodory metric
$d_{cc}$ on $\Heis^n$ because it is a natural metric associated to the
sub-Riemannian structure on $\Heis^n$, and also because---in some sense---it
is the most difficult metric on $\Heis^n$ to consider for this question. The
conclusion of Theorem~\ref{T1.5} holds also for other metrics on $\Heis^n$
which are bi-Lipschitz equivalent to $d_{cc}$, such as the Kor\'anyi metric
\eqref{dKor}. Note that results of Haj{\l}asz \cite{hajlaszGAFA},
\cite{hajlaszIsometric} imply that the answer to the Lipschitz density
question can depend on the choice of the metric in the target, even
within a bi-Lipschitz equivalence class.

Theorem~\ref{T1.5} and Proposition~\ref{T0} show an analogy between Sobolev
mappings into the Heisenberg groups and mappings into spheres. The analogy
would be even more complete, if we could answer Question \ref{Q1} in the
affirmative. In some sense $\Heis^n$ behaves like $\Sph^n$, but the analogy is
quite intricate.
Indeed,
smooth mappings are always dense in the space of Sobolev mappings into
a manifold which is diffeomorphic to $\R^{2n+1}$.
However, $\Heis^n$ is only homeomorphic to $\R^{2n+1}$. Although the
identity mapping from $\Heis^n$ to $\R^{2n+1}$ is locally Lipschitz,
the Heisenberg group is not even locally bi-Lipschitz homeomorphic to
$\R^{2n+1}$, as the Hausdorff dimension of $\Heis^n$ equals $2n+2$.
Thus one needs to look directly at the geometry of the Heisenberg
group which is very complicated.

In the case of Sobolev mappings into manifolds, the answer to the
density question depends, in particular, on homotopy groups of the
target. The homotopy groups of $\Heis^n$ are trivial, but a
more appropriate object to consider would be Lipschitz homotopy
groups. For example, the $n$th Lipschitz homotopy group of $\Heis^n$
is nontrivial \cite{BaloghF}, \cite{WengerY}, and this fact is
principally responsible for the lack of density. We discuss Lipschitz
homotopy groups and sub-Riemannian manifolds
in~Section~\ref{LH}.

\begin{remark}[Added in December 2012]
Wenger and Young \cite{wenger-young} further develop the theory
of Lipschitz homotopy groups of Heisenberg groups in relation to
Lipschitz extension problems. The related paper 
\cite{hajlasz-schikorra-tyson} contains an alternate proof of the 
nontriviality of the $n$th Lipschitz homotopy group of $\Heis^n$ and
other results of a similar nature, together with applications to
Lipschitz nondensity in Sobolev spaces. See also Remark
\ref{hst-remark}.
\end{remark}

In the proof of Proposition~\ref{T1} a Sobolev mapping which cannot be
approximated by Lipschitz mappings will be constructed as the
composition of the cavitation map \eqref{cavitation} with an explicit
bi-Lipschitz embedding of the sphere $\Sph^n$ into $\Heis^n$. The
existence of the latter is of independent interest.

\begin{theorem}\label{T2}
For any $n\geq 1$, there exists a bi-Lipschitz embedding of $\Sph^n$
into $\Heis^n$.
\end{theorem}

As observed by Balogh and F\"assler \cite{BaloghF} and Wenger and
Young \cite{WengerY}, the embedding from Theorem \ref{T2} admits no
Lipschitz extension from $\B^{n+1}$ to $\Heis^n$. Thus
$$
\pi_n^\lip(\Heis^n)\ne 0.
$$
To show that the mapping whose construction was described above cannot be
approximated by Lipschitz mappings, we will use a result of Ambrosio and
Kirchheim \cite{ambrosiok} and Magnani \cite{MagnaniThesis},
\cite{MagnaniRigidity} on the pure unrectifiability of $\Heis^n$. In the
proof of density in Theorem \ref{T1.5}(b), we will use recent results of
Wenger and Young \cite{WengerY}. Theorem~\ref{T2} is also related to
recent work on the Lipschitz extension problem with sub-Riemannian
target \cite{BaloghF}, \cite{RigotW} and on the construction of
Legendrian submanifolds of contact manifolds \cite{sullivan}.

\

Another issue is understanding what it means for a sequence of Sobolev
mappings $f_k\in W^{1,p}(M,\Heis^n)$ to converge to $f\in
W^{1,p}(M,\Heis^n)$.

The Heisenberg group $\Heis^n$ is homeomorphic to $\R^{2n+1}$ and
the identity mapping from $\Heis^n$ to $\R^{2n+1}$ is locally
Lipschitz. Hence if $f\in W^{1,p}(M,\Heis^n)$ is bounded, then
also $f\in W^{1,p}(M,\R^{2n+1})$. More generally, if $f\in W^{1,p}(M,\Heis^n)$
is not necessarily bounded, then still $f$ as a mapping into $\R^{2n+1}$ is
absolutely continuous on almost all lines.
Moreover the directional derivatives of $f$ are horizontal vectors
and hence for
$f\in W^{1,p}(\Omega,\Heis^n)$, $\Omega\subset\R^m$
we can define
\begin{equation}
\label{nabla-heis}
|\nabla f|_\Heis =
\left( \sum_{k=1}^m\left|\frac{\partial f}{\partial x_k}\right|_\Heis^2\right)^{1/2}\, ,
\end{equation}
where $|v|_\Heis$ stands for the length of the horizontal vector with respect to the
given metric in the horizontal distribution. See Section~\ref{SR} for
definitions.

\begin{theorem}
\label{T4}
Let $\Omega$ be a bounded domain in $\R^m$.
Suppose that
$f_k,f\in W^{1,p}(\Omega,\Heis^n)$, $k=1,2,\ldots$, $1\leq p<\infty$,
and $f_k\to f$ in $W^{1,p}(\Omega,\Heis^n)$. Then
$\int_{\{f_k-f \not\in Z\}} \left( |\nabla f_k|_\Heis^p+|\nabla
  f|_\Heis^p \right) \to 0$ as $k\to\infty$, where $Z$ denotes the
center of $\Heis^n$.
\end{theorem}

Clearly the theorem generalizes to the case in which $\Omega$ is
replaced by a compact manifold.

The condition is surprisingly strong. In particular it shows that $f_k$
{\em must} differ from $f$ by an element of the center $Z$ on a large
set. This phenomenon is quite unlike the case of manifold or Euclidean
targets. Similar phenomena have been previously observed in
\cite{hajlaszIsometric}.

Since bounded functions in $W^{1,p}(M,\Heis^n)$ belong also to
$W^{1,p}(M,\R^{2n+1})$ the proof of Theorem~\ref{T4} gives

\begin{corollary}\label{T5}
Let $M$ be a compact Riemannian manifold.
Suppose that $f_k,f\in W^{1,p}(M,\Heis^n)$, $k=1,2,\ldots,$
are uniformly bounded (i.e. the ranges of all the mappings are
contained in a fixed bounded subset of $\Heis^n$). If $f_k\to f$ in
$W^{1,p}(M,\Heis^n)$, then $f_k\to f$ in $W^{1,p}(M,\R^{2n+1})$.
\end{corollary}
Theorem~\ref{T4} shows, however, that the converse implication is
not true. Indeed, if $f_k,f\in W^{1,p}(M,\Heis^n)$ and $f_k\to f$ in
$W^{1,p}(M,\R^{2n+1})$, then it is very rarely true that $f_k\to f$ in
$W^{1,p}(M,\Heis^n)$.

\

The paper is organized as follows. In Section~\ref{SR} we recall the
definition and basic properties of sub-Riemannian manifolds and the
Heisenberg group. In Section~\ref{BE} we provide two different proofs
of Theorem~\ref{T2}. We find it important to present two different
approaches as they refer to different geometric structures underlying
$\Heis^n$. The first approach is based on the interpretation of
$\Heis^n$ as a conformal image of the boundary of the unit ball in
$\C^{n+1}$, punctured at one point, while the second is based on ideas
from symplectic geometry. In Section~\ref{LH} we discuss Lipschitz
homotopy groups. We do not prove any deep results there, but we think
that Lipschitz homotopy groups will eventually play an important role
in geometric analysis and geometric topology and we would like to
advertise the subject. In Section~\ref{SM} we define the class of
Sobolev mappings into metric spaces and in particular into the
Heisenberg group. We follow the presentation given in \cite{hajlaszt}.
In Section~\ref{P45} we prove Theorem~\ref{T4} and Corollary~\ref{T5},
and show that the class of Sobolev mappings into the Heisenberg group
defined in our paper agrees with that defined by Capogna and Lin
\cite{capognal}. Section~\ref{P13} contains the proofs of
Proposition~\ref{T1} and Theorem~\ref{T1.5}. The final
Section~\ref{sec:grushin} contains a variant of Theorem \ref{T1.5}
where the target space is replaced by the sub-Riemannian Grushin
plane.

\

\paragraph{\bf Acknowledgements.} We thank the referee for a careful
reading of the paper and for helpful remarks. In particular, we are
grateful to the referee for insisting that we include more detail in
the proof of Theorem \ref{LHT4}; this request led us to articulate
stronger conclusions regarding the structure of the Lipschitz homotopy
group $\pi_n^\lip(\Heis^n)$ than we had originally deduced. See Remark
\ref{ExtraRemark} for details.

\section{Sub-Riemannian geometry}
\label{SR}

\subsection{The Heisenberg group}

We represent the {\em Heisenberg group}  as the space
$\Heis^n=\C^n\times\R=\R^{2n+1}$ equipped with the group law
$$
(z,t)*(z',t')=\left(z+z',t+t'+2 \Imag \sum_{j=1}^n z_j
  \overline{z_j'}\right).
$$
The Heisenberg group $\Heis^n$ is a Lie group. A basis of left
invariant vector fields is given by
\begin{equation}\label{XY}
X_j=\frac{\partial}{\partial x_j} + 2y_j\frac{\partial}{\partial t},\
Y_j=\frac{\partial}{\partial y_j}-2x_j\frac{\partial}{\partial t},\qquad
j=1,\ldots,n,
\end{equation}
and $T=\frac{\partial}{\partial t}$.
Here and henceforth we use the notation
$$
(z,t) = (z_1,\ldots,z_n,t) = (x_1,y_1,\ldots,x_n,y_n,t).
$$
Note that $[X_j,Y_j]=-4T$, $j=1,\ldots,n$, while all other Lie
brackets of pairs of vectors taken from \eqref{XY} are equal to zero.
The Heisenberg group is equipped with the
{\em horizontal distribution} $H\Heis^n$, which is defined at every
point $p\in\Heis^n$ by
$$
H_p\Heis^n={\rm span}\, \{ X_1(p),\ldots,X_n(p),Y_1(p),\ldots,Y_n(p)\}.
$$
The distribution $H\Heis^n$ is equipped with a left invariant metric $\boldg$
such that the vectors 
$$
X_1(p),\ldots,X_n(p),Y_1(p),\ldots,Y_n(p)
$$
are orthonormal at every point $p\in\Heis^n$. We denote by $o=(0,0)$
the identity element in $\Heis^n$. A family of anisotropic dilations
$(\delta_r)_{r>0}$ on $\Heis^n$ is defined by
\begin{equation}\label{delta}
\delta_r(z,t) = (rz,r^2t), \qquad r>0.
\end{equation}

An absolutely continuous curve $\gamma:[a,b]\to\Heis^n$ is called {\em
horizontal} if $\gamma'(s)\in H_{\gamma(s)}\Heis^n$ for almost every $s$. The
Heisenberg group $\Heis^n$ is equipped with the {\em Carnot-Carath\'eodory
  metric} $d_{cc}$ which is defined as the infimum of the lengths of
horizontal curves connecting two given points. The length of the curve is
computed with respect to the metric $\boldg$ on $H\Heis^n$.
Let us remark in passing at this point that we will denote by $|\cdot|_\Heis$
the norm on the horizontal bundle induced by the metric $\boldg$, i.e.,
$$
|v|_\Heis = \boldg_p(v,v)^{1/2} \qquad \mbox{if $v\in H_p\Heis^n$.}
$$

It is well known that any two points in $\Heis^n$ can be connected by a
horizontal curve and hence $d_{cc}$ is a true metric. Actually, $d_{cc}$ is
topologically equivalent to the Euclidean metric. Moreover, for any compact
set $K$ there is a constant $C\geq 1$ such that
\begin{equation}
\label{SReq1}
C^{-1}|p-q|\leq d_{cc}(p,q)\leq C|p-q|^{1/2}
\end{equation}
for all $p,q\in K$.
In what follows $\Heis^n$ will always be regarded as the metric space
$(\Heis^n,d_{cc})$. It follows from \eqref{SReq1} that the identity mapping
from $\Heis^n$ to $\R^{2n+1}$ is locally Lipschitz, but its inverse is only
locally H\"older continuous with exponent $1/2$. The Hausdorff dimension of
$\Heis^n$ equals $2n+2$ since
$$
\cH^{2n+2}_{cc}(B_{cc}(p,r))=Cr^{2n+2}
$$
for all $p\in \Heis^n$ and $r>0$. Here $\cH^{2n+2}_{cc}$ stands for the
$(2n+2)$-dimensional Hausdorff measure with respect to $d_{cc}$ and
$B_{cc}(p,r)$ denotes a ball with respect to the Carnot-Carath\'eodory
metric. In fact, these Hausdorff measures are invariant with respect to left
translation and scale correctly under the anisotropic dilations $\delta_r$
defined above.

The variational problem which defines the Carnot-Carath\'eodory metric
always admits a solution: $(\Heis^n,d_{cc})$ is a geodesic
metric space. The geodesics can be written down explicitly in
parametric form. See, for example, the books by Montgomery
\cite{montgomery}, Bella{\"\i}che \cite{bellaiche},
Capogna--Danielli--Pauls--Tyson \cite{cdpt} or the paper of Marenich
\cite{marenich} for details. In fact, sub-Riemannian geodesics in
$\Heis^1$ are precisely the horizontal lifts of circular arcs in
$\R^2$ solving Dido's isoperimetric problem.

As a result, there is a rather complicated implicit formula for $d_{cc}$. We
will not need this formula in this paper. There are other metrics on
$\Heis^n$ which are bi-Lipschitz equivalent to $d_{cc}$, for instance,
the {\it Kor\'anyi metric} $d_K$ defined by
\begin{equation}\label{dKor}
d_K(p,q) = \Vert q^{-1} * p\Vert_K,
\end{equation}
where $\Vert (z,t)\Vert_K = (|z|^4+t^2)^{1/4}$.

On the Heisenberg group there is a natural horizontal gradient
\begin{equation}\label{nablaH}
\nabla_\Heis u = \sum_{j=1}^n (X_ju)X_j+(Y_ju)Y_j
\end{equation}
whose length with respect to the metric $\boldg$ on $H\Heis^n$ equals
$$
|\nabla_\Heis u|_\Heis = \sqrt{\sum_{j=1}^n |X_j u|^2+|Y_ju|^2}.
$$
Let $\Omega \subset \Heis^n$ be a domain. A function $u:\Omega\to\R$ is said to be
{\it continuously horizontally
differentiable} at a point $p\in\Omega$ if $X_ju$ and $Y_ju$ are
continuous at $p$ for all $1\le j\le n$.
We denote the class of functions on $\Omega$ which are $k$ times continuously
horizontally differentiable at each point of $\Omega$ by $C^k_\Heis(\Omega)$.

The center of the Heisenberg group is the vertical axis ($t$-axis)
$$
Z=\{(z,t)\in\Heis^n:\, z=0\}.
$$
It easily follows from the group law that, for $p,q\in\Heis^n$, we have
\begin{equation}\label{ZZ}
q^{-1}*p\in Z \quad \mbox{if and only if} \quad p-q\in Z.
\end{equation}
For $q\in \Heis^n$, let $d_q:\Heis^n\to\R$ be the function
$$
d_q(p) = d_{cc}(p,q).
$$
As with all distance functions on metric spaces, $d_q$ is
$1$-Lipschitz. By the Pansu--Rademacher differentiation theorem
\cite{pansu}, $d_q$ is horizontally differentiable at
almost every point of $\Heis^n$. However, a stronger result holds.

\begin{lemma}
\label{SR-T1}
For each $q\in\Heis^n$, $d_q$ is in
$C^\infty(\Heis^n \setminus \{ p : q^{-1} * p \in Z \})$.
\end{lemma}

Here the $C^\infty$ regularity refers to the underlying Euclidean structure
on $\R^{2n+1}$, but note that this is equivalent to $C^\infty$ regularity in
horizontal directions only. Indeed, if $g \in C^k_\Heis(\Omega)$ for
some $k\ge 1$, then it is easy to see that $g\in C^{\lfloor k/2\rfloor}(\Omega)$,
hence $C^\infty(\Omega) = C^\infty_H(\Omega)$.

Lemma \ref{SR-T1} was explicitly proved by Monti \cite{monti} in the case of
$\Heis^1$, and by Ambrosio-Rigot \cite{ambrosior} in the case of $\Heis^n$.
Note that Monti and Ambrosio-Rigot only state that $d_q$ is Euclidean $C^1$,
however, the proof easily extends to yield the improved regularity asserted in
Lemma~\ref{SR-T1}.

The preceding result together with the chain rule imply that if
$q_0\in\Heis^n$ and $f:(a,b)\to\Heis^n$ is a horizontal curve,
differentiable at $s_0$, and such that $q_0^{-1}*f(s_0)\not\in Z$, then
the function
$$
u(s)=d_{cc}(f(s),q_0)
$$
is differentiable at $s_0$ and
\begin{equation}
\label{monti-chain-rule}
u'(s_0) = \left\langle\nabla_\Heis d_{q_0}(f(s_0)),f'(s_0)\right\rangle_\Heis.
\end{equation}
Here, $\langle\cdot,\cdot\rangle_\Heis$ denotes the fixed and given metric defined on the
horizontal distribution of $\Heis^n$.

Monti also proved that the Carnot-Carath\'eodory distance function satisfies
the eikonal equation. See Theorem 3.8 in \cite{monti}. Expressed in the above
language, this reads

\begin{lemma}[Monti]
\label{eikonal}
For $p,q\in\Heis^n$ such that $q^{-1} * p \not \in Z$,
$|\nabla_\Heis d_q(p)|_\Heis = 1$.
\end{lemma}

This lemma is actually an easy consequence of Lemma~\ref{SR-T1}. Indeed, since the function
$d_q$ is $1$-Lipschitz, $|\nabla_\Heis d_q(p)|\leq 1$. On the other hand
if $\gamma(t)$ is a geodesic parametrized by arc-length connecting $q$ to $p$
and passing through $p$ at $t=t_0$, then $d_q(\gamma(t))=t$ and hence
$$
1= \left.\frac{d}{dt}\right|_{t=t_0} d_q(\gamma(t)) =
\left\langle\nabla_\Heis d_q(p),\gamma(t_0)\right\rangle_\Heis
\leq |\nabla_\Heis d_q(p)|_\Heis\, .
$$
\begin{remark}
\label{when-attained}
The above argument (due to Monti) shows that $|\nabla_\Heis d_q(p)|$
is attained as the directional derivative in a geodesic direction. We will
need this fact in the proof of Lemma~\ref{1Destimate}.
\end{remark}

\subsection{Geometric measure theory in the Heisenberg group}

The notion of rectifiability is fundamental in geometric measure theory. A
(countably) $k$-rectifiable set is one which can be well approximated, in a
Lipschitz sense, by subsets of $\R^k$ up to a set of Hausdorff $k$-measure
zero. Dual to this is the notion of unrectifiable set. A purely
$k$-unrectifiable set is one which contains no subset of positive Hausdorff
$k$-measure which is the Lipschitz image of a set in $\R^k$. For subsets of
Euclidean space, there is a nice dichotomy between these notions. We refer to
the book of Mattila \cite[Chapter 15]{Mattila} for details.

The notions can be extended to general metric spaces.

\begin{definition}
A metric space $(X,d)$ is {\it countably $k$-rectifiable} if there exists a
countable family of subsets $A_j \subset \R^k$ and Lipschitz maps $f_j:A_j\to
X$ so that $\cH^k(X \setminus \bigcup_j f_j(A_j)) = 0$.

A metric space $(X,d)$ is {\it purely $k$-unrectifiable} for some integer
$k\ge 1$, if $\cH^k(f(A))=0$ for all sets $A\subset\R^k$ and all
Lipschitz maps $f:A\to X$.
\end{definition}

Here $\cH^k$ denotes $k$-dimensional Hausdorff measure in $(X,d)$.

It turns out that rectifiability, defined as above, is of limited use
in sub-Riemannian spaces. One indication of this fact is the following
theorem of Ambrosio--Kirchheim and Magnani.
See \cite[Theorem 7.2]{ambrosiok}, \cite[Proposition
4.4.2]{MagnaniThesis} and \cite[Theorem 1.1]{MagnaniRigidity} and
compare \cite[Proposition 1 and Theorem 3]{BaloghF}.

\begin{theorem}[Ambrosio--Kirchheim, $n=1$; Magnani, arbitrary
  $n$]\label{AKM} For all $k\ge n+1$, the Heisenberg group $\Heis^n$
  is purely $k$-unrectifiable.
\end{theorem}

In other words, $\cH^{k}_{cc}(g(F))=0$ whenever ${g}$ is a Lipschitz map
from an subset $F \subset \R^k$ into $\Heis^n$ and $k\ge n+1$. Here
$\cH^k_{cc}$ stands for the $k$-dimensional Hausdorff measure in the metric
space $(\Heis^n,d_{cc})$.

One version of the Lipschitz extension problem asks for which pairs of
metric spaces $X$ and $Y$ it holds true that every partially defined
Lipschitz map from a subset of $X$ into $Y$ extends to a Lipschitz map
of all of $X$ into $Y$. We say that the pair $(X,Y)$ has the {\it
Lipschitz extension property} if there exists a constant $C\ge 1$ so
that every $L$-Lipschitz map $f:A\to Y$, $A\subset X$, has a
$CL$-Lipschitz extension $F:X\to Y$.

With some additional work, it follows from the pure
$(n+1)$-unrectifiability of $\Heis^n$ that
the pair $(\R^{n+1},\Heis^n)$ does not have the Lipschitz extension
property.
This was proved by Balogh and F\"assler \cite{BaloghF}, see also
Proposition~\ref{BaloghGeneralization} below. On the
other hand, we note the following
theorem of Gromov \cite{gromov-cc}.
Gromov's proof uses the deep machinery of microflexibility; a new proof
which avoids the use of this machinery (and extends the result to a more
general class of Carnot groups) has recently been given by
Wenger and Young \cite[Theorem~1.1]{WengerY}.

\begin{theorem}[Gromov, Wenger--Young]\label{WY}
Let $M$ be either a compact Riemannian $k$-manifold, with or without
boundary, or $M=\R^k$, where $k \le n$. Then the pair $(M,\Heis^n)$
has the Lipschitz extension property.
\end{theorem}

\subsection{Sub-Riemannian manifolds}\label{SRM}

Let $M$ be a smooth, connected manifold equipped with a distribution
$HM \subset TM$. We allow the possibility that $HM$ has nonconstant
rank, i.e., the function $p\mapsto\dim H_pM$ is not constant. (For an
example, see Section~\ref{sec:grushin}.)

For $i\ge 1$, let $H^i_pM$ be the subspace of $T_pM$ spanned by the values at
$p$ of all vector fields obtained as iterated commutators of length at most
$i$ of sections of $HM$. Thus
$$
H^2_pM = (HM\oplus[HM,HM])_p,
$$
$$
H^3_pM = (H^2M\oplus[HM,H^2M])_p,
$$
and so on. For fixed $p$, we obtain a flag of subspaces
$$
(0) =: H_p^0M \subset H_pM \subset H^2_pM \subset H^3_pM \subset
\cdots \subset T_p M.
$$

\begin{definition}
The pair $(M,HM)$ is said to satisfy the {\it bracket-generating property} if
there exists an integer $s<\infty$ so that $H_p^sM=T_pM$ for all $p\in M$. In
this case, we call $HM$ the {\it horizontal distribution}, $H_pM$ the {\it
horizontal tangent space} at $p$, and we call the minimal $s$ satisfying the
condition the {\it step} of the distribution.
\end{definition}

\begin{definition}
Let $(M,HM)$ satisfy the bracket-generating property, and let
$\boldg=(\boldg_p)$ be a smoothly varying family of inner products defined on
the horizontal tangent bundle $HM$. The triple $(M,HM,\boldg)$ is called a
{\it sub-Riemannian manifold}. Its {\it step} is the step of the distribution.
\end{definition}

A sub-Riemannian manifold $(M,HM,\boldg)$ is {\it regular} if the function
$$
p \mapsto (\dim H_pM,\dim H_p^2M,\dim H_p^3M,\cdots,\dim H_p^sM=\dim M)
$$
is constant on $M$. For example, the Heisenberg group $(\Heis^n,H\Heis^n,g)$
is a regular sub-Riemannian manifold of step two. More generally all Carnot
groups are regular sub-Riemannian manifolds. The Grushin plane (see
Section~\ref{sec:grushin}) is a non-regular sub-Riemannian manifold of
step two.

As we did for $\Heis^n$, we can define a Carnot-Carath\'eodory metric on any
sub-Riemannian manifold. An absolutely continuous curve $\gamma$ is called
{\it horizontal} if $\gamma'(s) \in H_{\gamma'(s)}M$ for a.e.\ $s$. The length
of a horizontal curve $\gamma$ is computed with respect to the metric $\boldg$ on
the horizontal bundle. Define a distance function $d_{cc}$ on $M$ by
infimizing the lengths of horizontal curves joining two given points. The
fundamental theorem of sub-Riemannian geometry is the Chow--Rashevsky
theorem. See, e.g., \cite[Theorem 2.2]{montgomery}.

\begin{theorem}[Chow, Rashevsky]
\label{Chow-Rashevsky}
Let $(M,HM)$ be bracket-generating. Then any two points in $M$ can be
connected by a horizontal curve. Consequently, on any sub-Riemannian manifold
$(M,HM,\boldg)$, $d_{cc}$ is a metric.
\end{theorem}

An estimate similar to \eqref{SReq1} holds. Let $M=(M,HM,\boldg)$ be a regular
sub-Riemannian manifold of step $s$. Let $\tilde{\boldg}$ be any Riemannian
metric on $M$. Then
for each compact $K\subset M$ there exists a constant $C$ so that the estimates
\begin{equation}
\label{SReq2}
d_{\tilde{\boldg}}(p,q) \le d_{cc}(p,q)\le Cd_{\tilde{\boldg}}(p,q)^{1/s}
\end{equation}
hold for all $p,q\in K$. See, for example, Nagel--Stein--Wainger \cite{NSW},
Gromov \cite{gromov-cc}, or Montgomery \cite[Theorem 2.10]{montgomery}.

The next notion is essential in the following section.

\begin{definition}
Let $(M,HM,\boldg)$ be a sub-Riemannian manifold, and let $N$ be a smooth
manifold. A {\it horizontal embedding} of $N$ into $M$ is a smooth embedding
$\phi:N\to M$ such that $d\phi:TN\to HM$.
\end{definition}

For instance, horizontal embeddings of intervals into $M$ are precisely
non self-intersecting horizontal curves.

\section{Horizontal bi-Lipschitz embeddings}\label{BE}

In this section we will construct bi-Lipschitz embeddings of the
sphere $\Sph^n$ into the Heisenberg group $\Heis^n$. According to
the following theorem, it suffices to construct horizontal
embeddings of $\Sph^n$. We will describe two different such
embeddings.

\begin{theorem}
\label{T-BE1}
Let $(M,HM,\boldg)$ be a sub-Riemannian manifold, let $N$ be a compact
manifold, and let $\phi$ be a smooth horizontal embedding of $N$
into $M$. Let $d_{\rm ext}$ be the restriction of the CC metric on $M$ to
$\phi(N)$ and let $d_{\rm int}$ be the Riemannian metric on $\phi(N)$
inherited from $\boldg$.
Then $d_{\rm ext}$ is bi-Lipschitz equivalent to
$d_{\rm int}$. More precisely, there exists $K\geq 1$ so that
$$
K^{-1}d_{\rm int}\leq d_{\rm ext}\leq d_{\rm int}
$$
on $\phi(N)$. Furthermore, any Riemannian metric on $\phi(N)$ is bi-Lipschitz equivalent
to $d_{\rm ext}$.
\end{theorem}

\begin{proof}
Since any two Riemannian metrics on a compact manifold are bi-Lipschitz
equivalent, the last part of the theorem immediately follows from the
bi-Lipschitz equivalence of $d_{\rm int}$ and $d_{\rm ext}$.

The inequality $d_{\rm ext}\leq d_{\rm int}$ follows immediately from the
definition of the two metrics. It remains to prove the other inequality.

Without loss of generality, we may assume that $N$ is a subset of $M$ and
$\phi$ is the inclusion map. Fix $\eps>0$ and consider the set
\begin{equation}\label{PQ}
\{(p,q) \in N\times N : d_{\rm ext}(p,q)\ge \eps\}.
\end{equation}
Define
$$
K_1=\sup\, \frac{d_{\rm int}(p,q)}{d_{\rm ext}(p,q)} \, ,
$$
where the supremum is taken over the set in \eqref{PQ}. Observe that
$$
K_1 \le \frac{\diam N}{\eps}
$$
is finite.

This reduces us to a local problem. We have to show that there exists
$K\geq 1$ such that
$$
K^{-1}\, d_{\rm int}(p,q)\leq d_{\rm ext}(p,q)
$$
provided $d_{\rm ext}(p,q)<\eps$ and $p\neq q$. The constant $\eps$ will be
fixed later.

Let $\dim M=n$ and $\dim N=k$. In a neighborhood $\cU$ of each point of $N$
there is a cubic coordinate system $\vi:\cU\to (-3,3)^n$ on $M$ with
coordinate functions $x_1,\ldots,x_n$ such that
$$
\cU\cap N =\{ x_{k+1}=0,\ldots,x_n=0\}\, .
$$
Since $N$ is compact, there is a finite family of coordinate systems
$(\cU_i,\vi_i)$, $i=1,2,\ldots,m$ as above, such that
\begin{equation}
\label{E-BE1}
N=\bigcup_{i=1}^m \vi^{-1}_i((-1,1)^n)\cap N\, .
\end{equation}
Let
$$
\eps_i=\inf \left\{ d_{cc}(r,s):\,
r\in \vi_i^{-1}([-1,1]^n),\ s\in M\setminus \vi_i^{-1}((-2,2)^n)\right\}\, ,
$$
and 
$$
\eps_0=\min \{ \eps_1,\ldots,\eps_m\}.
$$
Clearly $\eps_0>0$.

Let $p,q\in \vi^{-1}_i((-1,1)^n)\cap N$ be such that $d_{\rm ext}
(p,q)<\eps_0$. Let $\gamma$ be a horizontal curve connecting $p$ and $q$ of
length $\ell(\gamma)<2d_{\rm ext}(p,q)$. (Here length is computed with respect
to the sub-Riemannian metric $\boldg$ on $HM$.) Then $\gamma$ is contained in
$\vi^{-1}_i((-2,2)^n)$. Indeed, any horizontal curve connecting $p$ and $q$
and {\em not} contained in $\vi_i^{-1}((-2,2)^n)$ has length
$$
\ell(\gamma)\geq 2\eps_i\geq 2\eps_0>2 d_{\rm ext}(p,q)\, .
$$
Let $\delta>0$ be the Lebesgue number of the covering \eqref{E-BE1} with
respect to the metric $d_{\rm ext}$ and let $\eps=\min\{\eps_0,\delta\}$. Then
any two points $p,q\in N$ with $d_{\rm ext}(p,q)<\eps$ satisfy $p,q\in
\vi_i^{-1}((-1,1)^n)\cap N$ for some $i=1,2,\ldots,m$, furthermore,
any horizontal curve $\gamma$ connecting $p$ and $q$ of length
$\ell(\gamma)<2d_{\rm ext}(p,q)$ is contained in
$\vi^{-1}_i((-2,2)^n)$.

A simple compactness argument shows that there is a constant $K\geq 1$ such
that for any $i=1,2,\ldots,m$, any $x\in [-2,2]^n$ and any vector $v\in
H_{\vi_i^{-1}(x)}M$,
\begin{equation}\label{E-BE2}
K^{-1}\Vert d\vi_i(v)\Vert^2\leq
\boldg_{\vi_i^{-1}(x)}(v,v)\leq
K\Vert d\vi_i(v)\Vert^2\, .
\end{equation}
Here $\Vert\cdot\Vert$ denotes Euclidean length of a vector in $\R^n$.

As explained above, $d_{\rm ext}(p,q)$ is the infimum of the lengths
of horizontal curves connecting $p$ and $q$ and contained in
$\vi_i^{-1}((-2,2)^n)$. Denote by $\ell_{\R^n}(\vi_i\circ\gamma)$ the
Euclidean length of the curve $\vi_i\circ\gamma$ contained in $(-2,2)^n$. It
follows from \eqref{E-BE2} that
\begin{equation}
\label{E-BE3}
K^{-1/2}|\vi_i(p)-\vi_i(q)| \leq
K^{-1/2}\ell_{\R^n}(\vi_i\circ\gamma) \leq
\ell(\gamma) \leq
K^{1/2} \ell_{\R^n}(\vi_i\circ\gamma)\, ,
\end{equation}
so
\begin{equation}
\label{E-BE4}
d_{\rm ext}(p,q)\geq K^{-1/2}|\vi_i(p)-\vi_i(q)| \, .
\end{equation}
The line segment $\sigma$ joining $\vi_i(p)$ to $\vi_i(q)$ is contained in
the slice $\{ x_{k+1}=0,\ldots, x_n=0\}$ and hence
$\gamma=\vi_i^{-1}\circ\sigma$ is a (horizontal) curve in $N$.
Now \eqref{E-BE3} implies
\begin{equation}
\label{E-BE5}
d_{\rm int}(p,q) \leq \ell(\gamma) \leq
K^{1/2}\ell_{\R^n}(\sigma)=
K^{1/2}\vert\vi_i(p)-\vi_i(q)\vert\, .
\end{equation}
Finally \eqref{E-BE4} and \eqref{E-BE5} together yield 
$K^{-1}\, d_{\rm int}(p,q) \leq d_{\rm ext}(p,q)$ provided $d_{\rm
  ext}(p,q)<\eps$. The proof is complete.
\end{proof}

Next, we state a slightly stronger version of Theorem~\ref{T2}.

\begin{theorem}
\label{T-BE2}
For any $n\geq 1$, there is a horizontal and hence bi-Lipschitz embedding
$\phi:\Sph^n\to \Heis^n$.
\end{theorem}

In the next two subsections we present two different proofs of
Theorem~\ref{T-BE2}.

\subsection{Horizontal embeddings of spheres via several complex variables}\label{subsec:ComplexEmbedding}

The proof is split into several rather independent steps.

\subsubsection{CR structure}
The complex structure of $\C^{n+1}$ induces an operator $\cJ:\R^{2n+2}\to
\R^{2n+2}$ defined by
$$
\cJ[x_1,y_1,\ldots, x_{n+1},y_{n+1}]=
[-y_1,x_1,\ldots,-y_{n+1},x_{n+1}]\, .
$$
If $T_z\C^{n+1}$ is the {\em real} tangent space to $\C^{n+1}$ at $z$, then we
have the tangent operator $\cJ:T_z\C^{n+1}\to T_z\C^{n+1}$ given by
$$
\cJ\left(\sum_{j=1}^{n+1}\left(
a_j\frac{\partial}{\partial x_j}+b_j\frac{\partial}{\partial
  y_j}\right)\right) = \sum_{j=1}^{n+1}
\left( -b_j\frac{\partial}{\partial x_j}+a_j\frac{\partial}{\partial
    y_j}\right)\, .
$$
Let $\Omega\subset\C^{n+1}$ be a domain with smooth boundary. For
$z\in\partial\Omega$ let $T_z\partial\Omega$ be the real tangent space. We
define
$$
H_z\partial\Omega= T_z\partial\Omega\cap \cJ T_z\partial\Omega\, .
$$
This is the maximal complex subspace of $T_z\partial\Omega$, but we will
regard it as a real space. It easily follows from a dimension argument that
the real dimension of $H_z\partial\Omega$ is $2n$. We will call
$H\partial\Omega$ the {\em horizontal distribution} in the tangent bundle
$T\partial\Omega$. It is also known as the CR structure, but we will not refer
to CR structures in what follows.

\subsubsection{The Siegel domain.}
The Siegel domain $D$ is the set of all points $z=(z_1,\ldots,z_{n+1})\in\C^{n+1}$ such that $\Imag \, z_{n+1}>\sum_{j=1}^n |z_j|^2$. Identifying $(z_1,\ldots,z_{n+1}) \in \C^{n+1}$ with
$(x_1,y_1,\ldots,x_{n+1},y_{n+1}) \in \R^{2n+2}$, we observe that the
boundary $\partial D$ is defined by the equation $r = 0$, where
$r(z)=\sum_{j=1}^n (x_j^2+y_j^2)-y_{n+1}$. Hence
$$
v=\sum_{j=1}^{n+1}\left( a_j\frac{\partial}{\partial x_j} +
b_j\frac{\partial}{\partial y_j}\right)
$$
is in $T_z\partial D$ for some $z=(x_1,y_1,\ldots,x_{n+1},y_{n+1})\in\partial
D$ if $dr_z(v)=0$, which in turn is equivalent to
$b_{n+1} = 2\sum_{j=1}^n ( x_j a_j + y_j b_j)$.
Now $v\in H_z\partial D$ if also
$$
\cJ v = \sum_{j=1}^{n+1} \left(-b_j\frac{\partial}{\partial x_j} +
a_j\frac{\partial}{\partial y_j}\right)
$$
is in $T_z\partial D$, i.e., $a_{n+1}= 2\sum_{j=1}^n( -x_j b_j + y_j a_j)$.
Therefore the horizontal distribution $H\partial D$ is generated by the vector fields
$$
\widetilde{X}_j(z) =
\frac{\partial}{\partial x_j} + 2y_j\frac{\partial}{\partial x_{n+1}}
+2x_j\frac{\partial}{\partial y_{n+1}}\, ,
$$
and
$$
\widetilde{Y}_j(z) =
\frac{\partial}{\partial y_j} - 2x_j\frac{\partial}{\partial x_{n+1}}
+2y_j\frac{\partial}{\partial y_{n+1}}\, ,
$$
for $j=1,2,\ldots,n$.

\subsubsection{The Heisenberg group}
Let $\pi:\C^{n+1}\supset \partial D\to \C^n\times\R$ be the projection
onto the first $2n+1$ coordinates:
\begin{equation}\label{the-defn-of-pi}
\pi(z_1,\ldots,z_{n+1}) = (x_1,y_1,\ldots, x_{n+1})\, .
\end{equation}
Then $X_j=d\pi_z(\widetilde{X}_j)
= \partial/\partial x_j + 2y_j\,\partial/\partial x_{n+1}$ and
$Y_j=d\pi_z(\widetilde{Y}_j) = \partial/\partial
y_j-2x_j\,\partial/\partial x_{n+1}$ generate the horizontal
distribution in $\Heis^n$. The map $\pi$ is a diffeomorphism of
$\partial D$ onto $\C^n\times\R$ which maps $H\partial D$ onto
$H\Heis^n$. In this way we identify the boundary of the Siegel domain
with the Heisenberg group.

If $\psi:\Sph^n\to\partial D$ is a smooth and horizontal (with respect
to $H\partial D$) embedding of the sphere, then
$\pi\circ\psi:\Sph^n\to\Heis^n$ is a smooth and horizontal embedding
into the Heisenberg group. Thus to prove Theorem~\ref{T-BE2} it
remains to construct a smooth and horizontal embedding
$\psi:\Sph^n\to\partial D$.

\subsubsection{The unit ball}
Let $B=\{ z\in \C^{n+1}:\, \sum_{j=1}^{n+1} |z_j|^2<1\}$ be the unit ball in
$\C^{n+1}$. Define the horizontal distribution $H\partial B$ as before. The
boundary $\partial B$ is given by $r=0$, where $r(z)=\sum_{j=1}^{n+1}
(x_j^2+y_j^2)-1$, hence
$$
v=\sum_{j=1}^{n+1}\left( a_j\frac{\partial}{\partial x_j} +
b_j\frac{\partial}{\partial y_j}\right)
$$
is in $T_z\partial B$ if $dr_z(v)=0$, i.e.\ $\sum_{j=1}^{n+1} (a_j
x_j + b_j y_j)=0$.

Let $\R^{n+1}$ be the real subspace of $\C^{n+1}$ generated by the
coordinates $x_1,\ldots, x_{n+1}$. Then $\Sph^n=\partial
B\cap\R^{n+1}$ is the standard unit sphere.

\begin{lemma}
\label{T-BE3}
For any $z\in \Sph^n$, $T_z\Sph^n\subset H_z\partial B$, i.e. the sphere
$\Sph^n$ is horizontally embedded into $\partial B$.
\end{lemma}

\begin{proof}
As before, $v=\sum_{j=1}^{n+1} a_j\frac{\partial}{\partial x_j} \in
T_z \Sph^n \subset T_z\partial B$ if $\sum_{j=1}^{n+1} a_j x_j=0$. It
remains to show that $\cJ v\in T_{z}\partial B$. We have $\cJ v
=\sum_{j=1}^{n+1} a_j\frac{\partial}{\partial y_j}$. Since $y_j=0$ for
$j=1,\ldots,n+1$ we have $\sum_{j=1}^{n+1} a_j y_j=0$, so $\cJ v\in
T_z\partial B$.
\end{proof}

Note that the {\em south pole}
$p_0=(z_1,\ldots,z_n,z_{n+1})=(0,0,\ldots,-1)$ belongs to $\Sph^n$.
Let $R\in U(n+1)$ be a $\C$-linear rotation of $\C^{n+1}$ so that
$p_0 \not\in R(\Sph^n)$. Since $R$ maps $\C$-linear subspaces of
$T_p\C^{n+1}$ onto $\C$-linear subspaces of $T_{R(p)}\C^{n+1}$, it
preserves the horizontal distribution $H\partial B$. Thus $R(\Sph^n)$
is a horizontally embedded sphere in $\partial B$ which does not
contain the south pole.

\subsubsection{The Cayley transform}
The unit ball $B$ and the Siegel domain $D$ are biholomorphically equivalent
via the {\em Cayley transform} $C:B\to D$,
$$
C(z_1,\ldots,z_{n+1})=
\left(\frac{z_1}{1+z_{n+1}},\ldots, \frac{z_n}{1+z_{n+1}},
\boldi\left(\frac{1-z_{n+1}}{1+z_{n+1}}\right)\right)\, .
$$
The mapping $C$ extends to the boundaries. It has one singularity on the
boundary. Namely, it maps the south pole $p_0$ into a point at infinity. Since
$dC$ is $\C$-linear, it maps $H\partial B$ onto $H\partial D$, except at
$p_0$. Hence $C(R(\Sph^n))$ is a horizontally embedded sphere into $\partial
D$. If we use the explicit rotation $R\in U(n+1)$ given by
$R(z_1,\ldots,z_{n+1})=(z_1,\ldots,z_n,\boldi z_{n+1})$ and the map
$\pi$ defined in \eqref{the-defn-of-pi}, then we obtain the following
map $\phi=\pi\circ C \circ R$ from $\Sph^n$ to~$\Heis^n$:
\begin{equation}\label{specific-phi}
\phi(x_1,\ldots,x_{n+1}) = \left( \frac{x_1}{1+\boldi x_{n+1}},
  \ldots, \frac{x_n}{1+\boldi x_{n+1}}, \Imag \left( \frac{\boldi
      x_{n+1} - 1}{1 + \boldi x_{n+1}} \right) \right),
\end{equation}
for $(x_1,\ldots,x_{n+1}) \in \Sph^{n}$.

This completes the first construction for Theorem~\ref{T-BE2}.

\subsection{Horizontal embeddings of spheres via symplectic geometry}\label{subsec:symplectic-construction}

In this section we indicate another construction of a smooth horizontal
embedding of $\Sph^n$ into $\Heis^n$ arising from symplectic geometry. This
example has previously been considered by Eckholm, Etnyre and Sullivan
\cite[Example 3.1]{sullivan} and Balogh and F\"assler \cite[Section
4]{BaloghF}.

We begin by recalling some background on Lagrangian and Legendrian
embeddings. Consider the Euclidean space $\R^{2n}$ of dimension $2n$. We
denote points in $\R^{2n}$ by $(x_1,y_1,\ldots,x_n,y_n)$. Let $\omega =
\sum_{j=1}^n dx_j \wedge dy_j$ denote the standard symplectic form. We note
that $\omega$ is exact, since $\omega = \frac12 d\beta$, where $\beta = \sum_j
(x_j\,dy_j - y_j\,dx_j)$.

A smooth mapping $f$ from an $m$-dimensional manifold $M$ into $\R^{2n}$ is
called {\it Lagrangian} if $f^*\omega = 0$. If $f=(g_1,h_1,\ldots,g_n,h_n)$,
this condition reads
\begin{equation}
\label{lagrangian}
\sum_{j=1}^n dg_j \wedge dh_j = 0.
\end{equation}
In local coordinates $(u_1,\ldots,u_m)$ on $M$, \eqref{lagrangian} reads
$$
\sum_{\substack{k,\ell=1 \\ k<\ell}}^m
\left( \sum_{j=1}^n
\frac{\partial g_j}{\partial u_k}
\frac{\partial h_j}{\partial u_\ell} -
\frac{\partial g_j}{\partial u_\ell}
\frac{\partial h_j}{\partial u_k} \right) du_k \wedge du_\ell = 0.
$$
Now consider Euclidean space $\R^{2n+1}$ of dimension $2n+1$. We denote points
in $\R^{2n+1}$ by $(x_1,y_1,\ldots,x_n,y_n,t)$. We introduce the {\em
  contact form}
\begin{equation}\label{the-alpha-form}
\alpha = dt + 2 \sum_j (x_j \, dy_j - y_j \, dx_j) = dt + 2 \beta.
\end{equation}
Observe that $d\alpha = 4 \omega$.

A map $F:M\to\R^{2n+1}$ is called {\it Legendrian} if $F^*\alpha = 0$. If
$F=(f,\tau)$ this condition reads
\begin{equation}
\label{exactness}
d\tau + 2 f^*\beta = 0.
\end{equation}
In this case, we say that $F:M\to\R^{2n+1}$ is a {\it Legendrian lift} of
$f:M\to\R^{2n}$.

Observe that if $F$ is Legendrian, then $f$ is necessarily Lagrangian, because
$$
f^*\omega=\frac{1}{2}f^*(d\beta)=-\frac{1}{4}d^2\tau = 0.
$$

The following lemma is an immediate consequence of \eqref{exactness}.

\begin{lemma}
\label{lifting-lemma}
Let $f:M\to\R^{2n}$ be a smooth mapping.
Then there exists a Legendrian lift $F$ of $f$ if and
only if the $1$-form $f^*\beta$ is exact.
\end{lemma}

\begin{remark}
\label{Legendrian-horizontal-remark}
Identify $\R^{2n+1}$ with the Heisenberg group $\Heis^n$. The kernel of
$\alpha$ at a point $p$ is the horizontal subspace $H_p\Heis^n$. A map
$F:M\to\R^{2n+1}$ is Legendrian if and only if it is horizontal, as a map to
$\Heis^n$.
\end{remark}

We now recall Example 3.1 from \cite{sullivan}.

\begin{example}
\label{E1}
Consider the mapping $\tilde{f}:\R^{n+1}\to\R^{2n}$ given by
$$
\tilde{f}(x_0,x') = \tilde{f}(x_0,x_1,\ldots,x_n) =
(x_1,x_0x_1,\ldots,x_n,x_0x_n)\, ,
$$
i.e. $\tilde{f}=(g_1,h_1,\ldots,g_n,h_n)$, where
$g_j=x_j$, $h_j=x_0x_j$. Here $x'=(x_1,\ldots,x_n) \in \R^n$. The rank of the
derivative $d\tilde{f}$ equals $n+1$ everywhere except on the line $x'=0$. If we
restrict $\tilde{f}$ to the unit sphere $\Sph^n$, then, clearly, the rank of
the derivative of
\begin{equation}\label{definition-of-f}
f=\tilde{f}|_{\Sph^n}:\Sph^n\to\R^{2n}
\end{equation}
equals $n$ at all points different than $(\pm1,0,\ldots,0)$. However, one can
easily check that also at the points $(\pm1,0,\ldots,0)$ the derivative of
$\tilde{f}$ restricted to the tangent space of $\Sph^n$ has rank~$n$.

This is to say that the map $f:\Sph^n\to\R^{2n}$ is an immersion.
It is not an embedding since $f(\pm1,0,\ldots,0)=(0,\ldots,0)$.
However, $f$ becomes an embedding when restricted to
$\Sph^n\setminus \{ (-1,0,\ldots,0),(1,0,\ldots,0)\}$. Note that
$$
\tilde{f}^*\beta=
\sum_{j=1}^n\left(g_jdh_j-h_jdg_j\right) =
\left(\sum_{j=1}^n x_j^2\right)\, dx_0.
$$
The form $\tilde{f}^*\beta$ is not exact, not even closed. On the other hand,
the form
$$
(1-x_0^2)dx_0=d\left(x_0-\frac{x_0^3}{3}\right)
$$
is exact on $\R^{n+1}$ and coincides with $\tilde{f}^*\beta$ on $\Sph^n$,
since $1-x_0^2=\sum_{j=1}^n x_j^2$ on $\Sph^n$. If $\iota:\Sph^n\to\R^{n+1}$
is the identity map, then $f=\tilde{f}\circ\iota$ and
$$
f^*\beta=(\tilde{f}\circ\iota)^*\beta=
\iota^*(\tilde{f}^*\beta) =
\iota^*\left(d\left(x_0-\frac{x_0^3}{3}\right)\right)=
d\left(\left(x_0-\frac{x_0^3}{3}\right)\circ\iota\right)\, .
$$
Hence the form $f^*\beta$ is exact and thus the mapping $f:\Sph^n\to\R^{2n}$
has a Legendrian lift
\begin{equation}\label{definition-of-F}
F:\Sph^n\to\R^{2n+1}.
\end{equation}
If $F=(f,\tau)$, then \eqref{exactness} gives
$\tau=\frac{2}{3}x_0^3-2x_0+C$. The Legendrian lift $F$ is a smooth embedding,
since $\tau(-1,0,\ldots,0)\neq \tau(1,0,\ldots,0)$. Therefore $F$ is a
horizontal embedding of $\Sph^n$ into $\Heis^n$ and thus bi-Lipschitz.
\end{example}

This completes the second construction for Theorem~\ref{T-BE2}.

\begin{remark}\label{BF-remark}
Balogh and F\"assler \cite{BaloghF} show that the map $F:\Sph^n\to\Heis^n$
given in \eqref{definition-of-F} has no Lipschitz extension
$\tilde{F}:\B^{n+1}\to\Heis^n$. In fact, they show that any continuous
extension $\hat{f}:\B^{n+1}\to\R^{2n}$ of the map $f$ from
\eqref{definition-of-f} has the property that $\hat{f}(\B^{n+1})$ has
positive $(n+1)$-dimensional (Euclidean) Hausdorff measure $\cH^{n+1}$. If
$\tilde{F}$ were a Lipschitz extension of $F$,
and $\hat{f}$ its projection onto $\R^{2n}$,
then we would have $\cH^{n+1}_{cc}(\tilde{F}(\B^{n+1})) \ge
\cH^{n+1}(\hat{f}(\B^{n+1})) > 0$, but this contradicts the pure
$(n+1)$-unrectifiability of $\Heis^n$ as asserted in the theorem of
Ambrosio--Kirchheim and Magnani (see Theorem~\ref{AKM}). We generalize
this result below in Proposition~\ref{BaloghGeneralization}.
\end{remark}

\begin{remark}\label{parallelizability-remark}
A manifold $M$ is {\it stably parallelizable} if $M\times\R$ is
parallelizable, i.e., has trivial tangent bundle. Examples of stably
parallelizable manifolds include all orientable hypersurfaces (with or
without boundary) as well as all products of spheres. According to a
theorem of Gromov \cite[p.\ 61]{gromov-pdr}, every stably
parallelizable $n$-manifold $M$ admits a Legendrian embedding into
$\R^{2n+1}$. According to Remark~\ref{Legendrian-horizontal-remark},
such an embedding is horizontal when considered as a map to $\Heis^n$.
This observation leads to numerous other horizontal embeddings of
smooth manifolds into the Heisenberg group $\Heis^n$. The case of a
product of spheres, $M=\Sph^{k_1}\times\cdots\times\Sph^{k_r}$, can be
done explicitly using a variant of the immersion in Example~\ref{E1}.
\end{remark}

\section{Horizontal and Lipschitz homotopy groups}
\label{LH}

For Riemannian manifolds $M,N$, the density of $C^\infty(M;N)$ in
$W^{1,p}(M;N)$ depends on the topology of the two manifolds. For
example, for $1 \leq p < n+1$, $C^\infty(\B^{n+1},N)$ is dense in
$W^{1,p}(\B^{n+1},N)$ if and only if $\pi_{\lfloor p \rfloor}(N) = 0$.
See Bethuel \cite{bethuel1} and Hang-Lin \cite{hangl2}.

Theorem~\ref{T1.5} seems to contradict a
similar statement for maps $\B^{n+1}\rightarrow \Heis^n$, since
$\Heis^n \cong \R^{2n+1}$ is contractible and has $\pi_i(\Heis^n)=0$
for $i\geq 1$. However, classical homotopy groups do not take into
account the metric structure of $\Heis^n$. We propose two
modifications of the classical homotopy groups which may be
appropriate for consideration of the Lipschitz density problem in
sub-Riemannian manifolds and general metric spaces.

\begin{definition}
Let $(X,x_0)$ be a pointed metric space. We define {\em Lipschitz
homotopy groups}
$$
\pi_n^{\lip}(X,x_0)
$$
in the same way as classical homotopy groups, with the exception that
both the maps and homotopies are required to be Lipschitz. We
emphasize that we make no restriction on the Lipschitz constants. In
particular, we do not require that the optimal Lipschitz constant for
a homotopy between two pointed maps $f,g:(Q^n,\partial Q^n)
\to(X,x_0)$, where $Q^n = [0,1]^n$, agree with---or even be comparable
to---that of the maps $f$ and $g$.

Unlike the case for the classical homotopy groups, it is not
immediately clear whether the Lipschitz homotopy groups
$\pi_n^\lip(X,x_0)$ can equivalently be defined via homotopy classes
of maps $(\Sph^n,s_0) \to (X,x_0)$ for a basepoint $s_0 \in \Sph^n$.
Nevertheless, one can easily check that $\pi_n^\lip(X,x_0) = 0$ if and
only if every Lipschitz map $(\Sph^n,s_0) \to (X,x_0)$ admits a
Lipschitz extension $\B^{n+1} \to X$. In fact, if $\varphi:\Sph^n \to
X$ is a Lipschitz map with no Lipschitz extension
$\tilde\varphi:\B^{n+1} \to X$ and $H : Q^n \to \Sph^n$ is Lipschitz
with $H|_{\partial Q^n} = s_0$, then  $\varphi \circ H$ defines a
nonzero element of $\pi_n^\lip(X)$. We will use this observation in
what follows to show that various Lipschitz homotopy groups of
Heisenberg groups are or are not trivial. 

If $(Y,y_0)$ is a pointed sub-Riemannian manifold, we may instead
require the maps and homotopies to be smooth and horizontal. The
resulting horizontal homotopy groups are called {\em smooth horizontal
homotopy groups}; we denote them by
$$
\pi_n^\hor(Y,y_0).
$$
In order to guarantee that the addition in $\pi_n^\hor(Y,y_0)$ results in
a smooth mapping we require that the mappings in question are constant in
a neighborhood of $\partial Q^n$. Observe, however, that
any smooth and horizontal mapping
$$
f:(Q^n,\partial Q^n)\to (Y,y_0)
$$
is horizontally homotopic to a mapping that is constant in a neighborhood
of $\partial Q^n$ inside $Q^n$.
\end{definition}

The standard verification that homotopy groups are in fact groups
applies directly to $\pi_n^\lip(X,x_0)$ and, with smoothing, to
$\pi_n^\hor(Y,y_0)$. Theorem~\ref{homotopy-group-properties} is
derived via standard arguments (see e.g.\ \cite{Hatcher}). Recall that
sub-Riemannian manifolds are horizontally connected according to the
Chow-Rashevsky theorem (Theorem~\ref{Chow-Rashevsky}). It follows from
the proof of the Chow-Rashevsky theorem that horizontal curves
connecting given two points are piecewise smooth, but one can actually
connect any two points by smooth horizontal curves, see Gromov
\cite[1.2B]{gromov-cc}. This fact is needed in part (3) of the next
result.

\begin{theorem}\label{homotopy-group-properties}
Let $X$ be a metric space with $x_0, x_1 \in X$, and $Y$ a
sub-Riemannian space with $y_0, y_1 \in Y$.
\begin{enumerate}
\item $\pi_n^\lip(X,x_0)$ and $\pi_n^\hor(Y,y_0)$ are abelian for
  $n>1$.
\item If $x_0$ and $x_1$ are joined by a rectifiable curve, then
  $\pi_n^\lip(X,x_0) \cong \pi_n^\lip(X,x_1)$.
\item For any $y_0,y_1 \in Y$, $\pi_n^\hor(Y,y_0) \cong \pi_n^\hor(Y,y_1)$.
\end{enumerate}
\end{theorem}

If $X$ is rectifiably connected, we may refer to $\pi_n^\lip(X)$
without reference to basepoint as it is well-defined up to
isomorphism. Likewise, $\pi_n^\hor(Y)$ is well-defined up to
isomorphism.

\begin{theorem}\label{LHT3}
Let $Y$ be a Riemannian manifold. Then $\pi_n(Y)=\pi_n^\lip(Y) =
\pi_n^\hor(Y)$ for any $n\ge1$.
\end{theorem}

Indeed, any continuous map $f:Q^n\to Y$ is homotopic to a smooth one
and smooth maps $f,g:Q^n\to Y$ that are continuously homotopic are
smoothly homotopic.

Here is another situation where the Lipschitz and classical homotopy
groups agree.

\begin{example}\label{Miller-example}
By a result of Cannon--Conner--Zastrow \cite{Zastrow}, any continuous
map from $\Sph^n$, $n\geq 2$ into $\R^2$ is homotopic to a constant map within its
image. Put in other words, every planar set $X$ is {\it aspherical};
its homotopy groups $\pi_n(X)$ are trivial for all $n\ge 2$. According
to an unpublished result of Jason Miller
\cite[Remark 2.14]{hei:lipschitzlectures}, the same conclusion holds
for Lipschitz homotopy groups: $\pi_n^\lip(X,x_0) = 0 $ for all $n\ge
2$ and $x_0 \in X$ if $X\subset\R^2$.
\end{example}

However, Lipschitz homotopy groups can differ substantially from their
classical counterparts.

\begin{example}
Let $X$ be a metric space in which no two distinct points can
be connected by a rectifiable curve. Then $\pi_n^\lip(X,x_0)=0$ for
any $x_0\in X$ and $n\geq 1$. Indeed, any Lipschitz map from
$\Sph^n$ to $X$ must be constant. For example, if $X$ is a standard
closed von Koch snowflake, then $\pi_1^\lip(X,x_0)=0$. On the other
hand $\pi_1(X,x_0)=\Z$, since $X$ is homeomorphic to $\Sph^1$. Note
that there is no contradiction with the previous example, as both
$\pi_n^\lip(X,x_0)$ and $\pi_n(X,x_0)$ vanish for this (planar) set
$X$ whenever $n\ge 2$.
\end{example}

In the original preprint version of this paper, we posed the following
question. Wenger and Young \cite{wenger-young} answered Question
\ref{QQ} in the negative.

\begin{question}\label{QQ}
For a sub-Riemannian manifold $Y$, the inclusion $C^\infty(\Sph^n, Y)
\hookrightarrow \lip(\Sph^n,Y)$ induces a homomorphism $\pi_n^\hor(Y)
\rightarrow \pi_n^\lip(Y)$. Is it an isomorphism?
\end{question}

We now show that smooth horizontal embeddings of $\Sph^n$ into
$\Heis^n$ are not Lipschitz null-homotopic. We conclude that
$\pi_n^\lip(\Heis^n)$ and $\pi_n^\hor(\Heis^n)$ are nontrivial, even
though $\pi_k(\Heis^n)=0$ for $k\ge 1$.

\begin{proposition}\label{BaloghGeneralization}
Let $\phi:\Sph^n \hookrightarrow \Heis^n$ be any smooth horizontal
embedding. Then $\phi$ cannot be extended to a Lipschitz map
$\tilde\phi: \B^{n+1} \rightarrow \Heis^n$.
\end{proposition}

\begin{proof}
Balogh and F\"assler \cite{BaloghF} established the conclusion for the
specific map $\phi: \Sph^n \rightarrow \Heis^n$ described in
Subsection~\ref{subsec:symplectic-construction}.
Since in our situation we do not have an explicit representation of $\phi$,
the argument has to be more general.

Assume by way of contradiction that $\tilde{\phi}:\B^{n+1}\to\Heis^n$ is a
Lipschitz extension of $\phi$. According to the theorem of Ambrosio-Kirchheim
and Magnani,
$\cH^{n+1}_{cc}(\tilde{\phi}(\B^{n+1}))=0$. Since the identity map from
$\Heis^n$ to $\R^{2n+1}$ is locally Lipschitz we also have that the ($n+1$)-dimensional
Euclidean Hausdorff measure of $\tilde{\phi}(\B^{n+1})$ is zero.
Therefore it suffices to prove

\begin{proposition}\label{positive-measure}
If $\psi:\Sph^n\to\R^k$, $k\geq n+1$ is a smooth embedding and
$\tilde{\psi}:\B^{n+1}\to\R^k$ is a Lipschitz extension of $\psi$, then
$\cH^{n+1}(\tilde{\psi}(\B^{n+1}))>0$.
\end{proposition}

\begin{proof}
We will need a version of the Stokes theorem for Lipschitz maps.

\begin{lemma}\label{stokes}
If $g:\B^{n+1}\to\R^\ell$, $\ell\ge n$, is Lipschitz and
$\omega$ is a smooth $n$-form on $\R^\ell$, then
$$
\int_{\partial\B^{n+1}} g^*\omega = \int_{\B^{n+1}} g^*(d\omega)\, .
$$
\end{lemma}

\begin{proof}
Extend ${g}$ to $\tilde{g}$ on $\B^{n+1}(2)$ by letting
$$
\tilde{g}(x) = g\left(\frac{x}{|x|}\right)
\quad
\mbox{when $|x|\geq 1$.}
$$
Thus $\tilde{g}$ is a radial extension of ${g}$ from $\partial \B^{n+1}$
to $\B^{n+1}(2)\setminus \B^{n+1}(1)$ and $\tilde{g}=g$ on $\B^{n+1}(1)$.
Since $\tilde{g}$ on $\B^{n+1}(2)\setminus \B^{n+1}(1)$ is constant
in radial directions, the rank of $d\tilde{g}$ is at most $n$ on
$\B^{n+1}(2)\setminus \B^{n+1}(1)$. Thus for any $1<r<2$
$$
\int_{\B^{n+1}(r)\setminus \B^{n+1}(1)} \tilde{g}^*(d\omega) = 0\, ,
$$
so
$$
\int_{\B^{n+1}(r)} \tilde{g}^*(d\omega) = \int_{\B^{n+1}(1)}g^*(d\omega)\, .
$$
The mapping $\tilde{g}$ is Lipschitz continuous and hence
in any Sobolev space $W^{1,p}$. Fix $n+1<p<\infty$ and let $h_i$ be a standard
approximation of $\tilde{g}$ constructed with the help of convolution. Then
$h_i\to\tilde{g}$ in $W^{1,p}(\B^{n+1}(2))$. Since $h_i$ is smooth,
for any $0<r<2$ we have
$$
\int_{\B^{n+1}(r)} h_i^*(d\omega) =
\int_{\Sph^n(r)} h_i^*\omega
$$
by the classical Stokes' theorem. Since $h_i\to\tilde{g}$ in $W^{1,p}$,
applying Fubini's theorem we may select a subsequence (still denoted
by $h_i$) such that
$$
h_i|_{\Sph^n(r)}\to \tilde{g}|_{\Sph^n(r)}
\quad
\mbox{in $W^{1,p}(\Sph^n(r))$ for a.e. $0<r<2$,}
$$
see \cite[pp.\ 189-190]{hajlaszSobolev}.
Fix such $r$ in the interval $1<r<2$. Then H\"older's inequality
yields
$$
\int_{\B^{n+1}(r)} h_i^*(d\omega) \to
\int_{\B^{n+1}(r)} \tilde{g}^*(d\omega) =
\int_{\B^{n+1}(1)} g^*(d\omega).
$$
On the other hand, another application of H\"older's inequality yields
$$
\int_{\B^{n+1}(r)} h_i^*(d\omega) =
\int_{\Sph^n(r)} h_i^*\omega \to
\int_{\Sph^n(r)} \tilde{g}^*\omega =
\int_{\Sph^n(1)} g^*\omega,
$$
where the last equality follows from the fact that $\tilde{g}$
on $\Sph^n(r)$ is the rescaling of ${g}$ on $\Sph^n(1)$. Hence
$$
\int_{\B^{n+1}(1)} g^*(d\omega) = \int_{\Sph^n(1)} g^*\omega
$$
and the lemma follows.
\end{proof}

Now we complete the proof of Proposition \ref{positive-measure}. Let
$\omega$ be the pullback by $\psi^{-1}$ of the volume form of $\Sph^n$
to $\psi(\Sph^n)\subset\R^k$. We can assume that $\omega$ is a smooth
form. Indeed, we can extend it smoothly from $\psi(\Sph^n)$ to $\R^k$
using local extensions in coordinate systems and a partition of unity.
The extended form may vanish outside a small neighborhood of
$\psi(\Sph^n)$, but it does not matter. According to Lemma
\ref{stokes} we have
$$
\int_{\B^{n+1}} \tilde{\psi}^*(d\omega) =
\int_{\Sph^n} \psi^*\omega = \cH^n(\Sph^n)>0.
$$
This implies that the rank of $d\tilde{\psi}$ is $n+1$ on a set of positive measure
and hence
$$
|J_{\tilde{\psi}}| =
\sqrt{\det(d\tilde{\psi})^T(d\tilde{\psi})}>0
$$
on a set of positive measure. According to the area formula \cite{EG}
$$
\int_{\tilde{\psi}(\B^{n+1})}
\cH^0(\tilde{\psi}^{-1}(y))\, d\cH^{n+1}(y) =
\int_{\B^{n+1}} |J_{\tilde{\psi}}|>0
$$
where $\cH^0$ stands for the counting measure. This, however,
implies that the set $\tilde{\psi}(\B^{n+1})$ has positive
($n+1$)-dimensional measure. This completes the proof of
Proposition~\ref{positive-measure}.
\end{proof}

In view of the preceding discussion, the proof of
Proposition~\ref{BaloghGeneralization} is complete.
\end{proof}

\begin{remark}\label{ExtraRemark}
A modification of the proof yields the following stronger
version of Proposition \ref{BaloghGeneralization}: {\it Let
$\psi_1,\ldots,\psi_N:\Sph^n \hookrightarrow \Heis^n$ be smooth
horizontal embeddings such that 
\begin{equation}\label{ExtraRemarkEquation}
\psi_i(\Sph^n) \setminus \bigcup_{j\ne i} \psi_j(\Sph^n) \ne \emptyset
\qquad \mbox{ for some $i=1,\ldots,N$.}
\end{equation}
Then $\psi_1,\ldots,\psi_N$ are independent as elements of
$\pi_n^\lip(\Heis^n)$, i.e., no nontrivial relations among these
embeddings exist.} Indeed, suppose that some finite linear combination
$\sum_j m_j \psi_j$ was equal to zero in $\pi_n^\lip(\Heis^n)$. Since
Lipschitz homotopy groups were defined via mappings from the pair
$(Q^n,\partial Q^n)$, we interpret this assumption as
providing a Lipschitz map $\tilde\psi$ from $Q^{n+1}$ to $\Heis^n$
which is constant on all but one face of $Q^{n+1}$ and which is given
by the concatenation of $m_1$ copies of $\psi_1$, $m_2$ copies of
$\psi_2$, \dots, $m_N$ copies of $\psi_N$ on the remaining face. Let
$\psi:\partial Q^{n+1} \to \Heis^n$ denote the restriction of
$\tilde\psi$ to the boundary.

Fix $i$ so that \eqref{ExtraRemarkEquation} holds. Then $\psi_i^{-1}
\left( \psi_i(\Sph^n) \setminus \bigcup_{j\ne i} \psi_j(\Sph^n)
\right)$ has nonempty interior $V$ in $\Sph^n$ (since $\psi_i$ is an
embedding). Let $\omega$ be the pullback by $\psi^{-1}$ of a smooth
$n$-form supported in $V$ such that $\int_V \psi^*\omega > 0$.
The remainder of the proof proceeds as before.

Note in particular that if $N=2$, then the negation of
\eqref{ExtraRemarkEquation} implies that $\psi_1(\Sph^n) =
\psi_2(\Sph^n)$. Thus two embedded spheres are dependent as elements
of $\pi_n^\lip(\Heis^n)$ if and only if their images coincide, in
which case they agree up to a reparameterization of $\Sph^n$.
\end{remark}

\begin{theorem}\label{LHT4}
\begin{enumerate}
\item If $1\leq k<n$, then $\pi_k^\lip(\Heis^n)=0$.
\item $\pi_n^\lip(\Heis^n)$ is uncountably generated for $n>0$.
\end{enumerate}
\end{theorem}

\begin{proof}
First, we prove (1). Theorem~\ref{WY} immediately implies that every
Lipschitz map $f:\Sph^k\to\Heis^n$ admits a Lipschitz extension
$\tilde{f}:\B^{k+1}\to\Heis^n$ provided $1\leq k<n$. Hence
$\pi_k^\lip(\Heis^n)=0$.

There exists a bi-Lipschitz embedding $\phi:\Sph^n\to\Heis^n$ (see
Theorem~\ref{T-BE2}). Proposition~\ref{BaloghGeneralization}
guarantees that $\phi$ admits no Lipschitz extension
$\tilde{\phi}:\B^{n+1}\to\Heis^n$. Hence $\pi_n^\lip(\Heis^n)\neq 0$.
To see that $\pi_n^\lip(\Heis^n,x_0)$ has uncountably many generators
we need to show that there are uncountably many Lipschitz mappings
from $\Sph^n$ to $\Heis^n$ which map a fixed point on $\Sph^n$ to
$x_0$ and whose homotopy classes in $\pi_n^{\lip}(\Heis^n,x_0)$ have
no nontrivial relations. Applying a translation if necessary, we may
assume that
$$
\phi(-1,0,\ldots,0) = x_0 = (0,\ldots,0) \in \Heis^n.
$$
An uncountable family of maps as above is then obtained by considering
the maps $\phi_r:\Sph^n\to\Heis^n$ given by $\phi_r(x) =
\delta_r(\phi(x))$ for $r>0$, where $\delta_r:\Heis^n\to\Heis^n$ is
the dilation given in \eqref{delta}. The fact that there are no
relations between these mappings follows from the discussion in Remark
\ref{ExtraRemark}.
This completes the proof of (2).
\end{proof}

When $n>1$, we can also show that the higher homotopy groups of spheres arise
as subgroups of the horizontal homotopy groups of $\Heis^n$.

\begin{theorem}
\label{LHT5}
For $k\ge n\ge 2$, $\pi_k^H(\Heis^n)$ is nontrivial whenever $\pi_k(\Sph^n)$ is nontrivial.
\end{theorem}

We use the following lemma.

\begin{lemma}\label{LHT7}
Let $f: \R^m \rightarrow \Heis^n$ be smooth and horizontal. Then $df$ has rank at most $n$.
\end{lemma}

\begin{proof}
Without loss of generality assume $f$ is not constant. Suppose
that $\rank df_{x_0} = k$ is maximized for some $x_0 \in \R^m$. Since
$x\mapsto\rank df_x$ is lower semicontinuous, $\rank df_x = k$ for all
$x$ in some neighborhood $U$ of $x_0$.  We may choose a smooth
$k$-dimensional submanifold $N\subset U$ so that $df|_N$ is injective
and hence $f|_N:N\to\Heis^n$ is a smooth immersion. Restricting to a
submanifold $N'\subset N$ if necessary, we obtain a smooth
embedding $f|_{N'}:N'\to\Heis^n$. Then $f(N')$ is a smooth embedded
$k$-dimensional horizontal submanifold in $\Heis^n$, and so $k\le n$.
\end{proof}

\begin{proof}[Proof of Theorem~\ref{LHT5}]
In the following argument we use the specific map $\phi$ constructed
in Subsection~\ref{subsec:ComplexEmbedding}; see \eqref{specific-phi}
for a precise formula.

It is clear that the standard projection
\begin{equation}\label{standard-projection}
\Heis^n \cong \C^n \times \R \rightarrow \Re(\C^n)\times \R \cong
\R^{n+1}
\end{equation}
is injective on the image of $\phi$. Furthermore, up to a
diffeomorphism of $\R^{n+1}$ we may assume that $\pi\circ\phi(\Sph^n)$
is the unit sphere $\Sph^n \subset \R^{n+1}$.

Suppose that $f: \Sph^k \rightarrow \Sph^n$ is homotopically
nontrivial but that $\phi \circ f$ can be extended to a smooth
horizontal map $\tilde f: \B^{k+1} \rightarrow \Heis^n$. Projecting to
$\R^{n+1}$ as in \eqref{standard-projection}, write $\tilde F = \pi
\circ \tilde f$.

By Sard's Theorem, $\tilde F$ has full rank at the preimage of a.e.\
point. Lemma~\ref{LHT7} implies that the image of $\tilde F$ has
measure zero, so there exists a point in $\B^{n+1}$ which is not in
the image of $\tilde F$. Projecting radially from this point to
$\Sph^n = \partial \B^{n+1}$, we get (up to the diffeomorphism $\pi
\circ \phi$) an extension of $f$ to $B^{k+1}$. This contradicts the
nontriviality of $f$.

The preceding argument is summarized in the following diagram:
\begin{diagram}
\Sph^k  &\rTo^{f} & \Sph^n & \rInto^\phi & \Heis^n & \cong & \C^n
\times \R \\
\dInto  &         &        & \ruTo(4,2)^{\tilde f} & \dTo^{\pi} & & \\
B^{k+1} & &\rTo(0,4)^{\tilde F} & &\R^{n+1} &\cong
& \,\,\,\,\Real(\C^n)\times \R \\
&&&&\dTo\\
&&&&\Sph^n
\end{diagram}
We conclude that $\phi \circ f$ represents a nontrivial element in
$\pi_k^\hor(\Heis^n)$ whenever $f$ represents a nontrivial element in
$\pi_k(\Sph^n)$.
\end{proof}

\begin{remark}
Every dilation of $\phi\circ f$ also represents a nontrivial element
in $\pi_k^\hor(\Heis^n)$, and all such elements are pairwise not
Lipschitz homotopic. It follows that the horizontal homotopy groups of
$\Heis^n$ are also uncountably infinitely generated for $k \ge n$ so
that $\pi_k(\Sph^n)$ is nontrivial.
\end{remark}

Allcock \cite{allcock} proved a quadratic isoperimetric inequality in
the symplectic space $\R^{2n}$, $n\ge 2$. His result implies, in
particular, that every closed loop enclosing zero symplectic area can
be filled by an isotropic spanning disc. Lifting this result into the
Heisenberg group $\Heis^n$, $n\ge 2$, and translating it into the
language of Lipschitz homotopy groups leads to the conclusion
$$
\pi_1^\hor(\Heis^n)=0 \qquad \mbox{for $n\ge 2$.}
$$
Compare Theorem \ref{LHT4} (1). The following question remains open.

\begin{question}
Is $\pi_k^\hor(\Heis^n)=0$ for $1 \leq k<n$ and $n\ge 3$?
\end{question}

\begin{question}
Is $\pi_k^\hor(\Heis^1)=0$ for $k>1$?
\end{question}

\begin{question}\label{QQQ}
Is the homomorphism $\pi_k(\Sph^n) \rightarrow \pi_k^\lip(\Heis^n)$
induced by a bi-Lipschitz embedding $\phi:\Sph^n\to\Heis^n$ injective?
\end{question}

\begin{remark}\label{hst-remark}
Following preparation of this paper, Wenger and Young
\cite[Theorem 1]{wenger-young} surprisingly answered Question
\ref{QQQ} in the negative for certain $k$ and $n$. More
precisely, for $n+2\le k<2n-1$, they proved that every composition of
Lipschitz maps $\Sph^k\to\Sph^n\to\Heis^n$ extends to a Lipschitz map
$\B^{k+1}\to\Heis^n$. Applying this result to any homotopically
nontrivial map $\Sph^k\to\Sph^n$ composed with one of the bi-Lipschitz
embeddings $\Sph^n\to\Heis^n$ discussed above yields the indicated
counterexample to the result proposed in Question \ref{QQQ}. Wenger
and Young also proved, see \cite[Corollary 4]{wenger-young}, that
$\pi_k^\lip(\Heis^1)=0$ for $k\ge 2$.

For $n=2d$, the preceding result of Wenger--Young is valid in the
range $2d+2\le k\le 4d-2$. In \cite{hajlasz-schikorra-tyson}, two of
the authors (joint with Schikorra) employ a new version of the Hopf
invariant for low rank Lipschitz maps to obtain nontriviality of the
Lipschitz homotopy group $\pi_{4d-1}^\lip(\Heis^{2d})$, with
corresponding consequences for the Lipschitz nondensity in suitable
Sobolev spaces. The paper \cite{hajlasz-schikorra-tyson} also contains
an alternate proof of the nontriviality of $\pi_n^\lip(\Heis^n)$.
\end{remark}

\section{Sobolev mappings into metric spaces}\label{SM}
Every separable metric space $(X,d)$ admits an isometric embedding into a
Banach space. For example, given a dense subset $\{ x_i\}_{i=1}^\infty\subset
X$ and $x_0\in X$, the map
\begin{equation}\label{kuratowski}
\kappa:X\to\ell^\infty,
\quad
\kappa(x)=\left( d(x,x_i)-d(x_0,x_i)\right)_{i=1}^\infty
\end{equation}
is an example of such an isometric embedding. It is called the
{\em Kuratowski embedding}. An important property of $\ell^\infty$ is
that it is dual to a separable Banach space: $\ell^\infty
=(\ell^1)^*$. Thus every separable metric space can be isometrically
embedded into a Banach space which is dual to a separable Banach space.

Given an isometric embedding of $X$ into a Banach space $V$ (not necessarily
dual to a separable Banach space), and an open set $\Omega\subset\R^n$, we
define the Sobolev space of mappings $W^{1,p}(\Omega,X)$ as follows:
\begin{equation}\label{w1px}
W^{1,p}(\Omega,X)=\{ f\in W^{1,p}(\Omega,V):\, f(x)\in X\ \mbox{a.e.}\}.
\end{equation}
The vector valued Sobolev space $W^{1,p}(\Omega,V)$ is a well known object
and can be defined using the notion of weak derivatives. This requires the
notion of the Bochner integral.

If $V$ is any Banach space
and $A\subset\R^n$ is (Lebesgue) measurable, we say that $f\in L^p(A,V)$ if
\begin{enumerate}
\item[(1)] $f$ is {\em essentially separably valued}: $f(A\setminus
Z)$ is a separable subset of $V$ for some set $Z$ of Lebesgue measure zero,
\item[(2)] $f$ is {\em weakly measurable}: for every $v^*\in V^*$ with
  $\Vert v^*\Vert\leq 1$, $\langle v^*,f\rangle$ is measurable, and
\item[(3)] $\Vert f\Vert\in L^p(A)$.
\end{enumerate}
If $f\in L^1(A,V)$ we define the integral
$$
\int_A f(x)\, dx
$$
as an element of $V$ in the Bochner sense, see \cite[Chapter~5,
Sections~4-5]{yosida}, \cite{diestelu}. The Bochner integral has two important
properties:
$$
\left \langle v^*, \int_A f(x) \, dx \right \rangle = \int_A \langle
v^*, f(x) \rangle \, dx
$$
for every $v\in V^*$, and
$$
\left\Vert \int_A f(x) \, dx \right\Vert \le \int_A \Vert f(x)\Vert \, dx.
$$
We can now define the vector valued Sobolev space as follows. Let
$\Omega\subset\R^n$ be an open set and $V$ any Banach space (not necessarily
dual to a separable space). The Sobolev space $W^{1,p}(\Omega,V)$,
$1\leq p<\infty$, is defined as the class of all functions  $f\in
L^p(\Omega,V)$ such that for $i=1,2,\ldots,n$ there is $f_i\in
L^p(\Omega,V)$ such that
$$
\int_\Omega\frac{\partial\vi}{\partial x_i}\, f = - \int_\Omega \vi f_i
$$
for every $\vi\in C_0^\infty(\Omega)$, where the integrals are taken in the
sense of Bochner. Note that the integrands are supported on compact subsets of
$\Omega$. We denote $f_i=\partial f/\partial x_i$ and call these functions
{\em weak partial derivatives} of $f$. We also write $\nabla f =(\partial
f/\partial x_1,\ldots, \partial f/\partial x_n)$ and
\begin{equation}\label{gradient}
|\nabla f|=
\left( \sum_{i=1}^n \left\Vert\frac{\partial f}{\partial
      x_i}\right\Vert^2\right)^{1/2}\, .
\end{equation}
Sometimes we will write $|\nabla f|_V$ to emphasize the Banach space
with respect to which we compute the length of the gradient.
The space $W^{1,p}(\Omega,V)$ is equipped with the norm
$$
\Vert f\Vert_{1,p} = \left(\int_\Omega \Vert f\Vert^p\right)^{1/p} +
\left(\int_\Omega |\nabla f|^p\right)^{1/p}\, .
$$
It is easy to prove that $W^{1,p}(\Omega,V)$ is a Banach space. Using local
coordinate systems we can also easily define the Sobolev space $W^{1,p}(M,V)$
and hence $W^{1,p}(M,X)$ for any compact Riemannian manifold $M$.

Let $\kappa:(\Heis^n,d_{cc})\to\ell^\infty$ be the Kuratowski embedding used to
define the space $W^{1,p}(M,\Heis^n)$. Then by definition
$f\in W^{1,p}(M,\Heis^n)$ if and only if
$\bar{f}=\kappa\circ f\in W^{1,p}(M,\ell^\infty)$ and
$$
\Vert f\Vert_{1,p} =
\left( \int_M \Vert \bar{f}\Vert_\infty^p\right)^{1/p} +
\left( \int_M |\nabla\bar{f}|_{\ell^\infty}^p\right)^{1/p}\, .
$$
If $M=\Omega\subset\R^m$ is a bounded domain and
$f\in W^{1,p}(\Omega,\Heis^n)$, then
$$
\Vert f\Vert_{1,p} =
\left( \int_\Omega \Vert \bar{f}\Vert_\infty^p\right)^{1/p} +
\left( \int_\Omega |\nabla f|^p_\Heis\right)^{1/p}\, ,
$$
see Proposition~\ref{estimate}.

While the Sobolev space $W^{1,p}(M,V)$ can be defined for any Banach space
$V$, it has particularly nice properties if $V$ is dual to a separable Banach
space. Namely one can prove the following result. See e.g. \cite{hajlaszt}.

\begin{proposition}\label{SM-T1}
Let $\Omega\subset\R^n$ be an open set and let $V$ be dual to a separable
Banach space $Y$. Then $f\in W^{1,p}(\Omega,V)$ if and only if $f\in
L^p(\Omega,V)$ and the following two conditions hold:
\begin{enumerate}
\item[(i)] for every $v^*\in V^*$, we have $\langle v^*,f\rangle\in
  W^{1,p}(\Omega)$, and
\item[(ii)] there is a nonnegative function $g\in L^p(\Omega)$ such that
\begin{equation}
\label{majorant}
|\nabla \langle v^*,f\rangle|\leq g
\quad
\mbox{a.e.}
\end{equation}
for every $v^*\in V^*$ with $\Vert v^*\Vert\leq 1$.
\end{enumerate}
Moreover,
$$
\Vert f\Vert_p + \inf \Vert g\Vert_p\leq
\Vert f\Vert_{1,p}\leq   \Vert f\Vert_p + \sqrt{n}\inf \Vert g\Vert_p,
$$
where the infimum is over the class of all ${g}$ that satisfy \eqref{majorant}.
\end{proposition}

This result easily leads to the following characterization.
See \cite[Theorem~3.17]{HKST} and \cite{hajlaszt}.

\begin{proposition}\label{SM-T2}
Let $X$ be a separable metric space that is isometrically embedded into a
Banach space that is dual to a separable Banach space. Let $M$ be a compact
Riemannian manifold and $1\leq p<\infty$. Then $f$ is in $W^{1,p}(M,X)$ if and
only if $d(x_0,f)\in W^{1,p}(M)$ for every $x_0\in X$ and there is a
nonnegative function $g\in L^p(M)$ such that $|\nabla d(x_0,f)|\leq g$
a.e.\ for each~$x_0$.
\end{proposition}

The Sobolev spaces $W^{1,p}(M,X)$ behave well under Lipschitz
postcomposition. See Reshetnyak \cite[Corollary 1, p.\
580]{reshetnyak} for the following result in the case when $M=\Omega$
is a Euclidean domain. The general statement given here is similar.

\begin{proposition}\label{SM-T2.5}
Let $X$ and $Y$ be separable metric spaces that are isometrically
embedded into Banach spaces that are dual to separable Banach spaces.
Let $M$ be a compact Riemannian manifold and $1\leq p<\infty$. If
$F:X\to Y$ is Lipschitz, then the operation $f \mapsto F\circ f$ sends
$W^{1,p}(M,X)$ into $W^{1,p}(M,Y)$. Furthermore, if $F$ is
$L$-Lipschitz, then $|\nabla(F\circ f)| \le L |\nabla f|$ a.e.\ in
$\Omega$.
\end{proposition}

If $X$ is a separable metric space that is isometrically embedded
into a Banach space $V$ that is dual to a separable Banach space,
then $W^{1,p}(M,X)\subset W^{1,p}(M,V)$. It turns out that
$W^{1,p}(M,X)$ mappings can be approximated by $\lip(M,V)$ mappings,
but the real question is whether they can be approximated by
$\lip(M,X)$ mappings.

\begin{proposition}[\cite{hajlaszIsometric}]
\label{SM-T3}
If a Banach space $V$ is dual to a separable Banach space,
$M$ is a compact Riemannian manifold, possibly with boundary,
and $f\in W^{1,p}(M,V)$, $1\leq p<\infty$, then for every
$\eps>0$ there is $g\in \lip(M,V)$ such that
$|\{x:\, f(x)\neq g(x)\}|<\eps$ and $\Vert f-g\Vert_{1,p}<\eps$.
\end{proposition}
The proof of this result is similar to the proof of
Lemma~13 in \cite{hajlaszMathAnn}, so it was stated in
\cite{hajlaszIsometric} without proof. On the other hand there are
some technical differences and since later on we will need
arguments used in the proof, we decided to provide details.

Denote the volume measure on $M$ by $\mu$. For $h \in L^1(\mu)$ we
define the {\em Hardy-Littlewood maximal operator} by
$$
\cM h(x) = \sup_{r>0}\frac{1}{\mu(\B(x,r))}\int_{\B(x,r)}|h|\, d\mu.
$$
The following result is well known, see \cite{acerbif}, \cite{hajlaszMetric},
\cite[Theorems~3.2, 3.3]{hajlaszk}.

\begin{lemma}\label{SM-T4}
Given a compact Riemannian manifold $M$, possibly with boundary, and $1\leq
p<\infty$, there is a constant $C=C(M,p)$ such that
\begin{equation}
\label{SMeq1}
|u(x)-u(y)|\leq Cd(x,y)\left(\cM|\nabla u|(x)+\cM|\nabla u|(y)\right)
\quad
\mbox{a.e.\ $x,y\in M$}
\end{equation}
for all $u\in W^{1,p}(M)$.
\end{lemma}

More precisely, there is a set $F\subset M$ of measure zero such that
\eqref{SMeq1} holds for all $x,y\in M\setminus F$.

Let $V=Y^*$, where $Y$ is a separable Banach space and
let $u\in W^{1,p}(M,V)$. For every $v^*\in Y\subset Y^{**}=V^*$ with
$\Vert v^*\Vert_Y\leq 1$, $x\mapsto\langle v^*,u(x)\rangle\in W^{1,p}(M)$ and
$|\nabla \langle v^*,u\rangle|\leq |\nabla u|$. Hence
\begin{equation}
\label{SMeq2}
|\langle v^*, u(x)-u(y)\rangle| \leq
Cd(x,y)\left(\cM|\nabla u|(x)+\cM|\nabla u|(u)\right)
\quad
\mbox{a.e.\ $x,y$.}
\end{equation}
Note that the implicit exceptional set in \eqref{SMeq2} depends a priori on
$v^*$. However, let $D\subset Y$ be a countable and dense subset of the unit
ball in $Y$. Since \eqref{SMeq2} holds a.e.\ for every $v^*\in D$ and $D$ is
countable, there is a set $F\subset M$ of measure zero such that \eqref{SMeq2}
holds for all $v^*\in D$ and all $x,y\in M\setminus F$. Then
\begin{equation*}
\Vert u(x)-u(y)\Vert
 = \sup_{v^*\in D} |\langle v^*, u(x)-u(y)\rangle|
 \leq Cd(x,y)\left(\cM|\nabla u|(x)+\cM|\nabla u|(y)\right)
\end{equation*}
for all such $x,y$ and in particular, for a.e.\ $x,y$.

We proved the following result. See also \cite{HKST}.

\begin{lemma}
\label{SM-T5}
Let $M$ and $V$ be as in Proposition~\ref{SM-T3} and let $1\leq
p<\infty$. Then there is a constant $C=C(M,p)$ such that if $u\in
W^{1,p}(M,V)$, then
$$
\Vert u(x)-u(y)\Vert\leq
Cd(x,y)\left(\cM|\nabla u|(x)+\cM|\nabla u|(y)\right)
\quad
\mbox{a.e.\ $x,y$.}
$$
\end{lemma}

We will also need the following fact, which is Lemma~3.1 in
\cite{hajlaszIsometric}.

\begin{lemma}
\label{SM-T6}
Let $M$ and $V$ be as in Proposition~\ref{SM-T3} and $1\leq p<\infty$. If
$u_1,u_2\in W^{1,p}(M,V)$ and $u_1=u_2$ on a measurable set $E\subset M$, then
$\nabla u_1=\nabla u_2$ a.e. on $E$.
\end{lemma}

\begin{proof}[Proof of Proposition~\ref{SM-T3}]
Let $f\in W^{1,p}(M,V)$. We want to prove that $f$ can be approximated
by Lipschitz mappings as stated in Proposition~\ref{SM-T3}. Recall that
$$
\Vert f(x)-f(y)\Vert \leq
Cd(x,y)\left(\cM|\nabla f|(x)+\cM|\nabla f|(y)\right)
\quad
\mbox{a.e.\ $x,y$}
$$
by Lemma~\ref{SM-T5}. Let
$$
E_t=\{ x\in M:\, \cM |\nabla f|(x)\leq t\}.
$$
The set $E_t$ is compact and weak type estimates for the maximal operator
$\cM$ give
\begin{equation}
\label{SMeq3}
t^p\mu(M\setminus E_t)\to 0
\quad
\mbox{as $t\to\infty$.}
\end{equation}
Note that the mapping
$$
f|_{E_t}:E_t\to V
$$
is Lipschitz continuous with Lipschitz constant $2Ct$. Hence the function
$f|_{E_t}$ admits a $C't$-Lipschitz extension $f_t:M\to V$.
Indeed, there is a Whitney decomposition of $M\setminus E_t$ into balls and
subordinated Lipschitz partition of unity. This is to say there is a constant
$C\geq 1$ depending on the Riemannian structure of $M$ only and a sequence $\{
x_i\}_{i\in I}$ of points in $M\setminus E_t$ such that with
$r_i=\dist(x_i,E_t)/10$ we have
\begin{enumerate}
\item[(a)] $\bigcup_{i\in I} B(x_i,r_i)= M\setminus E_t$;
\item[(b)] $B(x_i,5r_i)\subset M\setminus E_t$ for all $i\in I$;
\item[(c)] for every $i\in I$ and all $x\in B(x_i,5r_i)$ we have
$5 r_i \leq \dist (x,E_t) \leq 15 r_i$;
\item[(d)] no point of $M\setminus E_t$ belongs to more than $C$ balls
$\{B(x_i,5r_i)\}_{i\in I}$;
\item[(e)] there is a family of Lipschitz continuous functions
$\{\vi_i\}_{i\in I}$ such that ${\rm supp}\, \vi_i\subset B(x_i,2r_i)$,
$0\leq \vi_i\leq 1$, $\sum_{i\in I}\vi_i= 1$ and the Lipschitz constant of
$\vi_i$ is bounded by $Cr_i^{-1}$.
\end{enumerate}
Now it is a routine verification to show that the function
$$
f_t(x)=
\left\{
\begin{array}{ccc}
f(x)    & \mbox{if $x\in  E_t$,}\\ \\
\sum_{i\in I}\left(\frac{1}{\mu(B_i)}\int_{B_i} f\, d\mu\right)\,
\vi_i(x) &   \mbox{if $x\not\in E_t$,}
\end{array}
\right.
$$
is the desired $C't$-Lipschitz extension of $f|_{E_t}$. For details, see for
example \cite[pp.\ 446--448]{hajlaszGAFA}, where the argument is presented in
a slightly different, but related setting.

Observe that if $t\geq 1$ is sufficiently large, then the norm of $f_t$ is
bounded by $Ct$. Indeed, let $x_0\in E_1$. Then for any $x\in M$ we have
$$
\Vert f_t(x)\Vert \leq
\Vert f_t(x)-f_t(x_0)\Vert + \Vert f(x_0)|\Vert \leq
Ctd(x,x_0) +\Vert f(x_0)\Vert
$$
and it suffices to observe that $d(x,x_0)\leq {\rm diam}\, M$ and take
$t\geq \Vert f(x_0)\Vert$. Thus
\begin{eqnarray*}
\left( \int_M \Vert f-f_t\Vert^p\right)^{1/p}
& \leq &
\left( \int_{M\setminus E_t} \Vert f\Vert^p\right)^{1/p} +
\left( \int_{M\setminus E_t} \Vert f_t\Vert^p\right)^{1/p} \\
& \leq &
\left( \int_{M\setminus E_t} \Vert f\Vert^p\right)^{1/p} +
C\left( t^p\mu(M\setminus E_t)\right)^{1/p}\to 0
\end{eqnarray*}
as $t\to\infty$. We applied here the weak type estimate \eqref{SMeq3}.
For derivatives we have a similar estimate
$$
\left(\int_M |\nabla (f-f_t)|^p \right)^{1/p} \leq
\left(\int_{M\setminus E_t} |\nabla f|^p\ \right)^{1/p} +
C\left(t^p\mu(M\setminus E_t)\right)^{1/p}\to 0
$$
as $t\to\infty$. Here we applied Lemma~\ref{SM-T6} and the weak type estimate
\eqref{SMeq3}. This completes the proof.
\end{proof}

\section{Sobolev convergence for $\Heis^n$-valued maps}
\label{P45}

Theorem~\ref{T4} 
is analogous to a similar phenomenon which holds in the
Euclidean context. See \cite{hajlaszIsometric}. Its proof
is based on an estimate relating the $\mbox{weak}^*$ derivative of the
difference between two Sobolev maps, postcomposed with the Kuratowski
embedding, to the norm of the gradients of the individual mappings.

We begin by recalling the notion of $\mbox{weak}^*$ derivative.

\begin{definition}
Let $V$ be a Banach space which is dual to a separable Banach space. Let
$f:[a,b]\to V$. A vector $v\in V$ is called a {\it $\mbox{weak}^*$ derivative}
of $f$ at $s\in [a,b]$ if $v$ is the $\mbox{weak}^*$ limit of the vectors
$h^{-1}(f(s+h)-f(s))$ as $h\to 0$.
\end{definition}

For the following lemma, see Lemmas~2.8 and~2.13 of \cite{hajlaszt}.
See also Lemma~2.3 in \cite{hajlaszIsometric}.

\begin{lemma}
Let $V$ be a Banach space which is dual to a separable Banach space $Y$ and
let $f:[a,b]\to V$ be absolutely continuous. Then the limit
$$
g(s) := \lim_{h\to 0} h^{-1} \Vert f(s+h)-f(s)\Vert
$$
exists for a.e.\ $s$, and $g\in L^1([a,b])$. Furthermore, for a.e.\ $s\in
[a,b]$ there exists a unique $\mbox{weak}^*$ derivative $f'(s)$ of $f$ at
$s$, and $\Vert f'(s)\Vert \le g(s)$.
\end{lemma}

In the case when $V=\ell^\infty$, the $\mbox{weak}^*$ derivative can be
computed explicitly. If $f=(f^i)_{i=1}^\infty$ is absolutely continuous from
$[a,b]$ to $\ell^\infty$, then
\begin{equation}\label{fs}
f'(s) = \biggl((f^i)'(s)\biggr)_{i=1}^\infty,
\end{equation}
whenever $f'(s)$ is defined. See, for example, \cite{hajlaszIsometric}.

Let $(p_i)_{i=1}^\infty$ be a dense subset of $\Heis^n$, $p_0\in\Heis^n$ and let
$\kappa:(\Heis^n,d_{cc})\to\ell^\infty$ be the Kuratowski embedding given by
$$
\kappa(p) =
\left( d_{cc}(p,p_i)-d_{cc}(p_0,p_i)\right)_{i=1}^\infty\, .
$$
If $f:[a,b]\to\Heis^n$ is absolutely continuous, then
$\bar{f}=\kappa\circ f:[a,b]\to\ell^\infty$ is absolutely continuous and
$$
\bar{f}'(s) = \left(\frac{d}{ds} d_{cc}(f(s),p_i)\right)_{i=1}^\infty
\quad
\mbox{a.e.}
$$

The following lemma is a variation on Lemma 2.6 in \cite{hajlaszIsometric}.

\begin{lemma}
\label{1Destimate}
Let $f,g:[a,b]\to\Heis^n$ be absolutely continuous, and let
$\kappa:(\Heis^n,d_{cc}) \to \ell^\infty$ be the Kuratowski embedding. Then
$\bar{f} = \kappa \circ f$ and $\bar{g} = \kappa \circ g$ are
absolutely continuous, and the $\mbox{weak}^*$ derivative
$(\bar{f}-\bar{g})':[a,b] \to \ell^\infty$ satisfies
$$
\Vert(\bar{f}-\bar{g})'(s)\Vert_\infty \ge \max \{ |f'(s)|_\Heis,
|g'(s)|_\Heis \} \ge \tfrac12 (|f'(s)|_\Heis + |g'(s)|_\Heis)
$$
for almost every $s\in [a,b]$ such that $f(s)^{-1} * g(s) \not\in Z$.
\end{lemma}

Note that $f'(s)$ is horizontal for a.e.\ $s$ whenever $f:[a,b]\to\Heis^n$ is
absolutely continuous.
\begin{proof}
It is easy to see that
\begin{equation}
\label{fs2}
(\bar{f}-\bar{g})'(s) =
\biggl( \tfrac{d}{ds} d_{cc}(f(s),p_i) - \tfrac{d}{ds} d_{cc}(g(s),p_i)
\biggr)_{i=1}^\infty
\end{equation}
for a.e.\ $s\in [a,b]$.

Let $S$ be the set of points $s\in [a,b]$ so that $f(s)^{-1} * g(s) \not \in
Z$. Fix $s\in S$ so that both $f$ and ${g}$ are differentiable at $s$ and
\eqref{fs2} holds. Almost every $s\in S$ is of this type.
If $f'(s) = g'(s) = 0$ the inequality is obvious. Assume then, without loss of
generality that $f'(s) \ne 0$ and $|f'(s)|_\Heis \ge |g'(s)|_\Heis$.

An application of Lemma~\ref{SR-T1}, or more specifically, of
\eqref{monti-chain-rule}, yields that
\begin{equation}
\label{qqq}
\tfrac{d}{ds} d_{cc}(f(s),p_i) =
\left\langle\nabla_\Heis d_{p_i}(f(s)),f'(s)\right\rangle_\Heis
\end{equation}
whenever $i$ is such that $p_i^{-1} * f(s) \not \in Z$.

If $L_\sigma g=\sigma*g$ is the left translation on $\Heis^n$, then
$dL_{-f(s)}(f(s)):H_{f(s)}\Heis^n\to H_o\Heis^n$. Hence
$$
t\mapsto \exp\left(tdL_{-f(s)}(f(s))f'(s)\right)
$$
is a geodesic passing through $o$ at $t=0$.
Here $\exp$ denotes the exponential map from the Lie
algebra of $\Heis^n$ to the group itself.
We are also using the fact that $\exp$ is the identity map,
\cite[p.~3]{FS} and that straight lines passing through $o$ and contained in
$\C^n\times\{ 0\}\subset\C^n\times\R=\Heis^n$ are geodesics,
\cite[Theorem~1]{marenich}. Thus
$$
\gamma(t)=f(s)*\exp\left(tdL_{-f(s)}(f(s))f'(s)\right)
$$
is a geodesic such that $\gamma(0)=f(s)$ and $\gamma'(0)=f'(s)$. Given
$\delta>0$, it follows from Remark~\ref{when-attained} that 
\begin{equation*}
-|f'(s)|_\Heis
=\left.\frac{d}{dt}\right|_{t=0} d_{cc}(\gamma(t),\gamma(\delta))
=\left\langle \nabla_\Heis
  d_{\gamma(\delta)}(f(s)),f'(s)\right\rangle_\Heis
=\frac{d}{ds}d_{cc}(f(s),\gamma(\delta))\, .
\end{equation*}
Given $\eps>0$ and $\delta>0$
it follows from the continuity of the derivative of $d_q$ that we can find
$p_i$ sufficiently close to $\gamma(\delta)$ so that
$$
\tfrac{d}{ds} d_{cc}(f(s),p_i) \le -|f'(s)|_\Heis + \eps.
$$

Similarly, we can find $p_j$ sufficiently close to $\gamma(-\delta)$ so that
\begin{equation}
\label{hot-dog}
\tfrac{d}{ds} d_{cc}(f(s),p_j) \ge |f'(s)|_\Heis - \eps.
\end{equation}

Since $d_{cc}$ is $C^\infty$ away from the center $Z$ (Lemma~\ref{SR-T1})
and $f(s)^{-1} * g(s) \not \in Z$, we may choose $p_i$ and $p_j$ so close to
$f(s)$ that
$$
\left| \tfrac{d}{ds} d_{cc}(g(s),p_i) - \tfrac{d}{ds} d_{cc}(g(s),p_j) \right|
< \eps.
$$
Hence one of the following quantities is greater than or equal to
$|f'(s)|_\Heis - 2\eps$:
$$
\left| \tfrac{d}{ds} d_{cc}(f(s),p_i) - \tfrac{d}{ds} d_{cc}(g(s),p_i) \right|
$$
or
$$
\left| \tfrac{d}{ds} d_{cc}(f(s),p_j) - \tfrac{d}{ds} d_{cc}(g(s),p_j) \right|
$$
Letting $\eps\to 0$ gives
$$
\Vert(\bar{f}-\bar{g})'(s)\Vert_\infty \ge |f'(s)|_\Heis = \max \{
|f'(s)|_\Heis, |g'(s)|_\Heis \}.
$$
This completes the proof of the lemma.
\end{proof}

Our next task is to replace the one-dimensional source space in
Lemma~\ref{1Destimate} with a domain in a higher-dimensional Euclidean space.
\begin{lemma}
\label{mDestimate}
Let $\Omega$ be a bounded domain in $\R^m$, let $f,g \in
W^{1,1}(\Omega,\Heis^n)$ and let
$\kappa:(\Heis^n,d_{cc}) \to \ell^\infty$ be the Kuratowski embedding. Then
$\bar{f} = \kappa \circ f$ and $\bar{g} = \kappa \circ g$ are in
$W^{1,1}(\Omega,\ell^\infty)$ and
\begin{equation}
\label{mDe}
|\nabla (\bar{f}-\bar{g})|_{\ell^\infty}
\ge \frac{1}{4} ( |\nabla f|_\Heis + |\nabla g|_\Heis) \,
\chi_{\{f-g \not\in Z\}} \qquad \mbox{a.e.}
\end{equation}
\end{lemma}

\begin{proof}
This follows from absolute continuity along almost all lines parallel to the
coordinate axes, the definition of the $\mbox{weak}^*$ derivative
and Lemma~\ref{1Destimate}.
\end{proof}

In order to take advantage of this inequality, we need to take care of the set
where the functions $f$ and $g$ differ by a point in the center $Z$. The
following lemma accomplishes this.

\begin{lemma}
\label{lemma6-6}
Let $\Omega$, $f$, $g$ be as in Lemma~\ref{mDestimate} and let $S$ be the set
of points $p\in \Omega$ for which $f(p) - g(p) \in Z$. Then
$\nabla f=\nabla g$ almost everywhere in $S$.
\end{lemma}

\begin{proof}
First, we consider the case $m=1$. Assume that $\Omega=[a,b]$ and
$f,g:[a,b]\to\Heis^n$ are absolutely continuous. Let $p\in S$ and suppose that
$f'(p)$ and $g'(p)$ exist and are horizontal, but are not equal. Then
\begin{equation}
\label{isolated}
(f(p+\eps) - g(p+\eps)) - (f(p)-g(p)) = \eps(f'(p)-g'(p)) +o(\eps).
\end{equation}
Since $f'(p)$ and $g'(p)$ are both horizontal, their difference cannot lie in
$Z$.
Hence for $\eps$ sufficiently small $f(p+\eps)-g(p+\eps)\not\in Z$,
because otherwise the left hand side of \eqref{isolated} would belong to $Z$ while the
right hand side would not.
It follows that $p$ is an isolated point of $S$. The set $E$ of such
points is at most countable.

Now we consider the general case. Assume that $f$ and $g$ are bounded
$W^{1,1}$ maps from $\Omega\subset\R^m$ to $\Heis^n$. Let $E$ denote the set
of points $p\in S$ so that $\nabla f(p)$ and $\nabla g(p)$ exist but are not
equal. For $p\in E$, there exists $i\in\{1,\ldots,m\}$ so that
$\tfrac{\partial}{\partial x_i}f(p)$ and $\tfrac{\partial}{\partial x_i}g(p)$
are not equal. Applying the argument from the previous paragraph, we find that
there exists a nontrivial line segment $L$ parallel to the $e_i$ axis and
passing through $p$, so that $L\cap S = \{p\}$. Denoting the set of such
points by $E_i$, we have $E=E_1\cup\cdots\cup E_m$. Repeated applications of
Fubini's theorem show that each of the sets $E_i$ has Lebesgue $m$-measure
zero in $\Omega$. Hence the measure of $E$ is equal to zero.
\end{proof}

\begin{proposition}
\label{estimate}
Let $\Omega$ be a domain in $\R^m$, let $f\in W^{1,1}(\Omega,\Heis^n)$
and let $\kappa:(\Heis^n,d_{cc})\to\ell^\infty$ be the Kuratowski embedding.
Then $\bar{f}=\kappa\circ f\in W^{1,1}(\Omega,\ell^\infty)$ and
$$
|\nabla \bar{f}|_{\ell^\infty} = |\nabla f|_\Heis
\quad
\mbox{a.e.}
$$
where
$|\nabla f|_\Heis$ was defined in \eqref{nabla-heis} and
on the left hand side of the inequality we have the norm defined in \eqref{gradient}.
\end{proposition}
\begin{proof}
First we consider the case $m=1$. Assume that $f:[a,b]\to\Heis^n$ is absolutely continuous.
Then at a point where $\bar{f}'(s)$ exists we have
$$
\Vert \bar{f}'(s)\Vert_\infty =
\sup_i\left|\frac{d}{ds} d_{cc}(f(s),p_i)\right|\, .
$$
Since
$$
\left|\frac{d}{ds} d_{cc}(f(s),p_i)\right|
\le
\liminf_{h\to 0} \frac{d_{cc}(f(s+h),f(s))}{|h|}
\le
\liminf_{h\to 0}\frac{1}{h}\int_s^{s+h} |f'(\tau)|_\Heis\, d\tau = |f'(s)|_\Heis
$$
for a.e. $s\in [a,b]$ we have $\Vert \bar{f}'(s)\Vert_\infty \leq
|f'(s)|_\Heis$. On the other hand inequality \eqref{hot-dog}
implies that $\Vert\bar{f}'(s)\Vert_\infty \geq |f'(s)|_\Heis$. Hence
$\Vert \bar{f}'(x)\Vert_\infty = |f'(s)|_\Heis$ a.e. The general case
$m\geq 1$ follows from the case $m=1$ and absolute continuity of $f$
along almost all lines parallel to the coordinate axes.
\end{proof}

\begin{proof}[Proof of Theorem~\ref{T4}]
Suppose that $\Omega$ and the maps $f_k,f$ are as in the statement of the theorem,
and assume that $f_k$ converges to $f$ in $W^{1,p}(\Omega,\Heis^n)$.
Using \eqref{mDe} we find
$$
\int_{\{f_k-f\not\in Z\}} (|\nabla f_k|_\Heis + |\nabla f|_\Heis)^p \le
C \int_\Omega
|\nabla(\bar{f}_k-\bar{f})|_{\ell^\infty}^p \to 0
$$
as $k\to\infty$. This completes the proof.
\end{proof}

\begin{proof}[Proof of Corollary~\ref{T5}]
By working with charts we may assume that $M=\Omega$ is a bounded
domain in $\R^m$. Assume
that $f_k$ converges to $f$ in $W^{1,p}(\Omega,\Heis^n)$.
Clearly $f_k\to f$ in $L^p(\Omega,\R^{2n+1})$.
For each $k$, let $S_k := \{ p\in \Omega : f(p) - f_k(p) \in Z \}$. By
Theorem~\ref{T4},
$$
\int_{\Omega\setminus S_k} |\nabla f_k|^p+|\nabla f|^p \leq
C \int_{\Omega\setminus S_k} |\nabla f_k|_\Heis^p + |\nabla f|_\Heis^p \to 0.
$$
The inequality follows from the fact that the ranges of all mappings $f_k,f$
are contained in a common bounded subset of $\Heis^n$ and the Heisenberg norm
$|v|_\Heis$ and the Euclidean norm $|v|$ are comparable for horizontal
vectors $v\in H_p\Heis^n$ when $p$ varies over a bounded set in $\Heis^n$.
By Lemma~\ref{lemma6-6} we have
$$
\int_{S_k} |\nabla(f_k-f)|^p = 0.
$$
Note that here we refer to the Euclidean difference and the Euclidean norm on
$\nabla(f_k-f)$. Combining these two statements completes the proof.
\end{proof}

\begin{remark}\label{koranyi-remark}
Everything discussed in this section works also for the Kor\'anyi metric $d_K$
on $\Heis^n$ (see \eqref{dKor} for the definition). The proofs are
easier, since $d_K$ is clearly smooth on the set $\{(p,q) \in
\Heis^n\times\Heis^n : p\ne q\}$. In particular, there is no need to
make any special distinction for pairs of points $p,q\in\Heis^n$ with
$p\ne q$ but $q^{-1} * p \in Z$.
\end{remark}

\subsection{The definition of Capogna and Lin}
\label{capogna-lin-remark}
We conclude this section by discussing the equivalence of the above
definition for the Sobolev space $W^{1,p}(M,\Heis^n)$ with that
of Capogna and Lin. This equivalence is already known and can be deduced,
for instance, by passing through the
definition of Korevaar and Schoen. Nevertheless, we find it instructive to give
a self-contained, direct discussion. In order to avoid issues with the
integrability of the function itself, we restrict attention to bounded
functions. Again by working in charts we may assume that $M=\Omega$ is
a Euclidean domain.

To prove the equivalence with the definition of Capogna and Lin it suffices to prove
the following result, because the condition stated in the result is equivalent to their
definition of the Heisenberg valued Sobolev space.
See Theorems 2.11 and 2.15 in \cite{capognal}.
\begin{proposition}
A bounded function
$f=(z,t)=(x_1,y_1,\ldots,x_n,y_n,t):\Omega\to\Heis^n$ lies in
$W^{1,p}(\Omega,\Heis^n)$ if and only if $f$ is an element of the
usual Sobolev space $W^{1,p}(\Omega,\R^{2n+1})$ and satisfies the
contact equation
\begin{equation}
\label{CE}
\nabla t = 2 \sum_{j=1}^n ( y_j \nabla x_j - x_j \nabla y_j)
\end{equation}
almost everywhere in $\Omega$.
\end{proposition}
\begin{proof}
To prove the ``only if'' statement, we consider a bounded function
$f\in W^{1,p}(\Omega,\Heis^n)$. Combining the fact that the identity
map from $\Heis^n$ to $\R^{2n+1}$ is Lipschitz on bounded sets with
Proposition~\ref{SM-T2.5} shows that $f\in W^{1,p}(\Omega,\R^{2n+1})$.
Moreover, the ACL property of $f$ (see, for instance, the proof of
Theorem 2.2 in \cite{hajlaszSobolev}) yields that the image of almost
every line parallel to a coordinate axis in $\Omega$, is a horizontal
curve. Consequently, the contact equation \eqref{CE} is satisfied
a.e.\ on such lines. By Fubini's theorem, \eqref{CE} holds a.e.\ in
$\Omega$.

To prove the ``if'' statement we reverse the argument. Consider a
bounded function $f \in W^{1,p}(\Omega,\R^{2n+1})$ which satisfies
\eqref{CE} almost everywhere. We first note that $f$ is ACL as a map
from $\Omega$ to $\R^{2n+1}$. From \eqref{CE}, by an application of
Fubini's theorem, we conclude that the image of almost every line
parallel to a coordinate axis in $\Omega$ is a horizontal curve. Along
horizontal curves in $\Heis^n$, the Hausdorff $1$-measures computed
with respect to the Euclidean and Carnot-Carath\'eodory metrics are
mutually absolutely continuous. Hence $f$ is ACL as a map to $\Heis^n$
and so also as a map to $\kappa(\Heis^n) \subset \ell^\infty$. It now
follows that the $\mbox{weak}^*$ partial derivatives of $f$ exist
almost everywhere and are bounded above (almost everywhere) by an
$L^p$ function. By Proposition~\ref{SM-T1}, $f$ lies in
$W^{1,p}(\Omega,\Heis^n)$.
\end{proof}

\section{Density and nondensity of Lipschitz maps in $W^{1,p}(M,\Heis^n)$}
\label{P13}

\subsection{Proof of Proposition~\ref{T1}}

Let $\psi:\Sph^n\to\Heis^n$ be a bi-Lipschitz embedding whose existence
was established in Theorem~\ref{T2}. We can also assume that $\psi$ is a
$C^\infty$ embedding as a map from $\Sph^n$ to $\R^{2n+1}$. Let
$$
f(x)=\psi\left(\frac{x}{|x|}\right)\in W^{1,p}(\B^{n+1},\Heis^n),
\quad
1\leq p<n+1.
$$
We will prove that the mapping $f$ cannot be approximated by Lipschitz mappings
$\lip(\B^{n+1},\Heis^n)$ when $n\leq p<n+1$.
To the contrary suppose that there is a sequence $g_k\in \lip(\B^{n+1},\Heis^n)$
such that
$$
g_k\to f
\quad \mbox{in $W^{1,p}(\B^{n+1},\Heis^n)$.}
$$
Clearly $f\in W^{1,p}(\B^{n+1},\R^{2n+1})$ and $g_k\in \lip(\B^{n+1},\R^{2n+1})$.
However, we cannot use Corollary~\ref{T5}, because we cannot assume that the mappings
$g_k$ are uniformly bounded.

By Theorem~\ref{AKM}, $\cH^{n+1}_{cc}(g_k(\B^{n+1})) = 0$ for all $k$.
Consequently, $\cH^{n+1}(g_k(\B^{n+1}))=0$ for all $k$ as well, since
the identity map from $\Heis^n$ to $\R^{2n+1}$ is locally Lipschitz.
(Recall that $\cH^k$ denotes the $k$-dimensional Euclidean Hausdorff
measure.)

Consequently, in order to arrive at a contradiction, it suffices to show that
\begin{equation}
\label{H0}
\cH^{n+1}(g_k(\B^{n+1}))>0 \qquad \mbox{for sufficiently large $k$.}
\end{equation}

Let $\omega$ be the pullback by $\psi^{-1}$ of the volume form from
$\partial \B^{n+1}=\Sph^n$ to $\psi(\Sph^n)$. We can assume that
$\omega$ is a smooth compactly supported $n$-form on $\R^{2n+1}$. We
used a similar construction in the proof of
Proposition~\ref{BaloghGeneralization}. 

By $\Sph^n(r)$ and $\B^{n+1}(r)$ we will denote the sphere and the ball of
radius $r$ centered at the origin. Observe that since $f|_{\Sph^n(r)}$ is just
a rescaling of the map $\psi$, then
$$
\int_{\Sph^n(r)} f^*\omega=\cH^n(\Sph^n)>0.
$$
Let $K={\rm supp}\, \omega$, let $S_k=\{x\in \B^{n+1}:\,
g_k(x)-f(x)\in Z\}$ and let $E_k=S_k\cup g_k^{-1}(K)$. We claim that
$$
(\nabla g_k)\chi_{E_k} \to \nabla f
\quad
\mbox{in $L^p(\B^{n+1})$.}
$$
Indeed, according to Lemma~\ref{lemma6-6}, $\nabla g_k=\nabla f$ a.e. in $S_k$
and hence
$$
\int_{S_k}|\nabla f-\nabla g_k|^p=0.
$$
Since the mappings $f$ and $g_k|_{g_k^{-1}(K)}$, $k=1,2,\ldots$
are uniformly bounded Theorem~\ref{T4} yields
$$
\int_{\B^{n+1}\setminus S_k} |\nabla f|^p + |\nabla g_k|^p\chi_{E_k} \leq
C\int_{\B^{n+1}\setminus S_k} |\nabla f|_\Heis^p+|\nabla g_k|_\Heis^p \to 0.
$$
Thus
\begin{equation*}
\int_{\B^{n+1}} |\nabla f-(\nabla g_k)\chi_{E_k}|^p
\leq
C\left( \int_{S_k} |\nabla f-\nabla g_k|^p +
\int_{\B^{n+1}\setminus S_k} |\nabla f|^p+|\nabla g_k|^p\chi_{E_k}\right)\to 0.
\end{equation*}
It follows from the Fubini theorem (again, see \cite[pp.\
189-190]{hajlaszSobolev}) that we may select a subsequence of $(g_k)$,
still denoted $(g_k)$, such that 
$$
(\nabla g_k)\chi_{E_k}|_{\Sph^n(r)} \to
\nabla f|_{\Sph^n(r)}
\quad
\mbox{in $L^p(\Sph^n(r))$ for a.e.\ $0<r<1$.}
$$
Fix such $r$. Since $K={\rm supp}\,\omega$,
$g_k^*\omega = 0$ in $\B^{n+1}\setminus g_k^{-1}(K)$ and hence
$$
g_k^*\omega = 0
\quad
\mbox{in $\B^{n+1}\setminus E_k$.}
$$
It follows from Lemma~\ref{stokes} that
$$
\int_{\B^{n+1}(r)} g_k^*(d\omega) =
\int_{\Sph^n(r)} g_k^*\omega =
\int_{\Sph^n(r)\cap E_k} g_k^*\omega,
$$
and an application of H\"older's inequality
(note that $g_k^*\omega,f^*\omega \in L^{p/n}(\Sph^n(r))$ and
$\tfrac{p}{n} \ge 1$ by assumption) yields
$$
\int_{\Sph^n(r)\cap E_k} g_k^*\omega
\to \int_{\Sph^n(r)} f^*\omega = \cH^n(\Sph^n)>0.
$$
Hence
$$
\int_{\B^{n+1}(r)} g_k^*(d\omega)>0
$$
for sufficiently large $k$.
Using the same argument involving the area formula as in the proof of
Proposition~\ref{BaloghGeneralization} we conclude that
$$
\cH^{n+1}(g_k(\B^{n+1}))\geq \cH^{n+1}(g_k(\B^{n+1}(r)))>0
$$
which proves \eqref{H0}. This completes the proof of Proposition~\ref{T1}.
\hfill $\Box$

\begin{proof}[Proof of Theorem~\ref{T1.5}(a)]
Suppose that $\dim M\geq n+1$. We need to prove that Lipschitz mappings
$\lip(M,\Heis^n)$ are not dense in $W^{1,p}(M,\Heis^n)$ when $n\leq
p<n+1$. This result will follow from Proposition~\ref{T1} in a
standard way; the same technique has already been used in the case of
maps into manifolds in \cite{bethuelz} and \cite{hajlaszSobolev}. For
the sake of completeness we provide the details.

It is easy to construct a smooth mapping $f:\B^{n+1}\to \Sph^n$ with two
singular points such that $f$ restricted to small spheres centered at the
singularities has degree $+1$ and $-1$ respectively and $f$ maps a
neighborhood of the boundary of the ball $\B^{n+1}$ into a point. We can model
the singularities on the radial projection mapping as in Proposition~\ref{T1}
so the mapping $f$ belongs to $W^{1,p}$. Let now $g:\B^{n+1}\times \Sph^{\dim
  M-n-1}\to \Sph^{n}$ be defined by $g(b,s)=f(b)$. We can embed the torus
$\B^{n+1}\times \Sph^{\dim M-n-1}$ into the manifold $M$ and extend the
mapping on the completion of this torus as a mapping into a point. Clearly
$g\in W^{1,p}(M,\Sph^{n})$.

Now let $\psi:\Sph^{n}\to\Heis^n$ be a smooth and horizontal (hence
bi-Lipschitz) embedding. We will prove that the mapping $\psi\circ g\in
W^{1,p}(M,\Heis^n)$ cannot be approximated by Lipschitz mappings from
$\lip(M,\Heis^n)$. By way of contradiction suppose that $u_k\in
\lip(M,\Heis^n)$ converges to $\psi\circ f$ in the norm of $W^{1,p}$. In
particular $u_k\to \psi\circ g$ in $W^{1,p}(\B^{n+1}\times \Sph^{\dim
M-n-1},\Heis^n)$. It follows from Fubini's theorem that there is a
subsequence of $(u_k)$, still denoted $(u_k)$, such that for a.e.\ 
$s \in \Sph^{\dim M - n - 1}$, $u_k$ restricted to the slice 
$\B^{n+1}\times\{s\}$ converges to the corresponding restriction of
$\psi\circ g$ in the Sobolev norm. Take such a slice and denote it
simply by $\B^{n+1}$. Hence $u_k$ restricted to a ball (of dimension
$n+1$) centered at the singularity of degree $+1$ converges to
$\psi\circ g$ restricted to the same ball. This is, however impossible
as was demonstrated in Proposition~\ref{T1}. 
\end{proof}

\begin{proof}[Proof of Theorem~\ref{T1.5}(b)]
Let $M$ be a compact Riemannian manifold of dimension at most $n$ and
let $f\in W^{1,p}(M,\Heis^n)$. Then $\bar{f}=\kappa\circ f\in 
W^{1,p}(M,\ell^\infty)$. According to Lemma~\ref{SM-T5},
$$
\Vert \bar{f}(x)-\bar{f}(y)\Vert \leq
Cd(x,y) \left(\cM|\nabla \bar{f}|(x)+\cM|\nabla \bar{f}|(y)\right)
\qquad \mbox{for a.e.\ $x,y$.}
$$
Let $E_t=\{ x\in M:\, \cM|\nabla \bar{f}|(x)\leq t\}$. The map
$\bar{f}|_{E_t}:E_t\to\ell^\infty$ is $Ct$-Lipschitz, so
$f|_{E_t}:E_t\to\Heis^n$ is also. According to Theorem~\ref{WY} there 
is a $C't$-Lipschitz map $f_t:M\to\Heis^n$ which coincides with $f$ on
$E_t$. By the same argument as in the proof of
Proposition~\ref{SM-T3}, $f_t\to f$ in $W^{1,p}(M,\Heis^n)$.
\end{proof}


\section{Sobolev maps into Grushin planes}\label{sec:grushin}

The Grushin plane $G_n$ is the sub-Riemannian manifold whose underlying space
is $M=\R^2$, with horizontal distribution $HM$ defined by the vector fields
$$
\xi_1 = \partl{}{x}, \qquad \xi_2 = x^n\partl{}{y}.
$$
We note that this distribution does not have constant rank, however, it is
bracket generating: after $n$ brackets we obtain $[\xi_1,[\xi_1,[\ldots
,\xi_2]\cdots] = n! \partl{}{y}$. We define a sub-Riemannian
(Carnot-Carath\'eodory) metric $d_{cc}$ on $M$ by declaring $\xi_1$ and
$\xi_2$ to be an orthonormal basis at each point $(x,y) \in M$ with $x\ne
0$, and declaring $\xi_1$ to be normal at each point $(0,y) \in
M$. We obtain a sub-Riemannian manifold $G_n$ of step $n+1$.

In this section we prove the following theorem.

\begin{theorem}\label{GrushinLipschitzDensity}
Equip $G_n$ with the Carnot-Carath\'eodory metric $d_{cc}$. Then Lipschitz
mappings from $\B^2$ to $G_n$ are not dense in $W^{1,p}(\B^2,G_n)$
provided $1\le p <2$.
\end{theorem}

For simplicity we consider only the case $G=G_1$. For $y_1>0$, $G$
contains infinitely many geodesics (local length minimizers) joining
$(0,0)$ and $(0,y_1)$. See, e.g., \cite[p.\ 275]{CalinChang}. They are
given by
\begin{equation}\label{GGF}
x_m(t) = \pm \sqrt{\frac{2y_1}{m \pi}} \sin(m \pi t), \;\;
y_m(t) = y_1\left(s-\frac{\sin(2 m \pi s)}{2 m \pi} \right)
\end{equation}
for $m\in \Z^+$. When $m=1$ we obtain a single arc rising to $(0,y_1)$
in the first quadrant and another in the second quadrant.


Concatenating these two geodesics for $y_1=1$ yields a set which is
bi-Lipschitz parameterized by $\Sph^1$. Let $\phi: \Sph^1 \rightarrow
G$ be a bi-Lipschitz mapping whose image is this union of geodesics.
Let $u_0: \B^2 \rightarrow \Sph^1$ be the cavitation map as in
\eqref{cavitation}. Since $u_0 \in W^{1,p}(\B^2,\Sph^1)$ for $p \in
[1,2)$, the map
\begin{equation}\label{F}
f = \phi \circ u_0
\end{equation}
is in $W^{1,p}(\B^2,G)$ for the same range of $p$. We will show that
$f$ cannot be approximated by Lipschitz functions from $\B^2$ to $G$.

The analog of the center $Z$ of the Heisenberg group is the vertical line
$Y:=\{(x,y) \in G:x=0\}$. The key technical result in the proof of the theorem
is the following lemma. Note that the restriction of the Hausdorff
$2$-measure in the metric $d_{cc}$ to the set $Y$ coincides (up to a
constant) with the Euclidean length measure on $Y$.

\begin{lemma}\label{GrushinMeasureZero}
Let $f: \B^2 \rightarrow G$ be a Lipschitz map. Then $\cH^2_{cc}(f(\B^2)
\cap Y)=0$.
\end{lemma}

\begin{proof}
Let $X = f^{-1}(Y) \subset \B^2$. As a substitute for the horizontal
(Pansu) differential, we will use Kirchheim's metric differential
\cite{kirch}. 
It follows from Kirchheim's metric
differentiability theorem \cite[Theorem 2]{kirch} that $f$ is almost
everywhere metrically differentiable. That is, at a.e.\ point in
$\B^2$ there exists a seminorm $mdf_x := \Norm{\cdot}_x$ such that
$d(f(x),f(y)) - \Norm{y-x}_x = o(\norm{y-x}$ as $y\to x$.

For a seminorm $s$ on $\R^2$, the {\it Jacobian} is defined as 
\[\mathbf{J}_2(s) = \frac{\pi}{\cH^2(\left\{x \mst s(x) \leq 1 \right\}
  )}.\]
By the area formula \cite[Theorem 5.1]{ambrosiok},
\[\int_X \mathbf{J}_2(mdf_x) \, dx = \int_Y \cH^0(f^{-1}(y)) \,
d\cH^2_{cc}(y).\]
Hence $\cH_{cc}^2(f(\B^2) \cap Y) \leq \int_X \mathbf{J}_2(mdf_x)$ and we
claim that
$$
\mathbf{J}_2(mdf_x)=0 \qquad \mbox{for a.e.\ $x\in X$.}
$$
To see why this claim holds, note that the identity map from $G$ to $\R^2$ is
locally Lipschitz, so $f$ is still Lipschitz when considered as a map into
$\R^2$. By Rademacher's theorem, the differential $df$ exists a.e. At points
where $df$ exists, it is clear that $\mathbf{J}_2(mdf_x) = \det(df_x)$. It
remains to show that $df$ is singular for a.e.\ $x\in X$.

Suppose $df_x$ is non-singular at $x \in X$ and let $v :=
(df_x)^{-1}(\partl{}{y})$. Then $|f(x+tv)-f(x)-(0,t)|_E$ is $O(t^2)$,
whence
$$
d_{cc}(f(x+tv),f(x)+(0,t)) = O(t)
$$
as $t\to 0$, by \eqref{SReq2}. On the other hand, there exists a constant
$c>0$ so that
$$
d_{cc}(f(x), f(x)+(0,t)) = c \sqrt{t}
$$
for all $t>0$, since $x\in X$ and $d_{cc}|_Y$ is a multiple of
$\sqrt{d_E|_Y}$. We conclude that
$$
d_{cc}(f(x+tv),f(x)) \ge c\sqrt{t} - O(t)
$$
as $t\to 0$, which contradicts the assumption that $f$ is Lipschitz.

In conclusion, $df_x$ is defined a.e.\ and is singular where defined.
Using the area formula, the proof is complete.
\end{proof}

We next reduce the question of Sobolev convergence to one about Euclidean
targets.

\begin{lemma}\label{GrushinToEuclid} If $f_k \rightarrow f$ in
  $W^{1,p}(\B^2,G)$, then $f_k \rightarrow f$ in $W^{1,p}(\B^2,\R^2)$.
\end{lemma}

\begin{proof}
The same method applies here as in the case of Heisenberg targets. The only
difference is that for $p,q \in G$, the phrase ``$p$ and $q$ are vertically
separated'' must be interpreted ``$p,q \in Y$''. Since the metric is
Riemannian outside of $Y$, it is $C^\infty$ on $G\setminus Y$.
\end{proof}

\begin{proof}[Proof of Theorem~\ref{GrushinLipschitzDensity}]
Let $f$ be the map in \eqref{F}. Suppose there exists a
sequence of Lipschitz maps $f_k:\B^2 \rightarrow G$ approximating $f$
in $W^{1,p}(\B^2,G)$. By Lemma~\ref{GrushinToEuclid}, $f_k$ converges
to $f$ in $W^{1,p}(\B^2,\R^2)$. Fix $\epsilon>0$ sufficiently small
(to be determined later) and fix $k$ sufficiently large so that
$$
\Norm{f-f_k}_{W^{1,p}(\B^2,\R^2)}<\epsilon.
$$

Let $g=f-f_k$. By Fubini's theorem, $g\vert_{S_r} \in
W^{1,p}(S_r,\R^2)$ for almost every $r \in (0,1]$, where $S_r$ denotes
the circle of radius $r$ centered at the origin. Furthermore, there
exists $r \in (1-2^{-p},1)$ so that the Sobolev norm of
$g\vert_{S_r}$ is less than $2\epsilon$. We claim that $f_k(B_r)$
intersects $Y$ in a set of positive length (and hence positive
$\cH^2_{cc}$ measure). Here $B_r$ denotes the disc of radius $r$
centered at the origin.

Since $\Vert g\Vert_{L^p(S_r)} < 2\epsilon$, there exists $x_0 \in S_r$ with
$|g(x_0)| < 2\epsilon(2\pi r)^{-1/p}$. Furthermore, $\Vert
dg\Vert_{L^p(S_r)} < 2\epsilon$. Applying the Fundamental Theorem of
Calculus along $S_r$ gives
$$
\Vert g\Vert_{L^\infty(S_r)} \le |g(x_0)| + \Vert dg\Vert_{L^1(S_r)}
\le 2\epsilon (
(2\pi r)^{-1/p} + (2\pi r)^{1-1/p} ).
$$
Since $r$ is bounded away from zero, it is clear that by choosing
$\epsilon$ sufficiently small (depending only on $p$) we may ensure
that
$$
\Vert f_k-f\Vert_{L^\infty(S_r)} = \Vert g\Vert_{L^\infty(S_r)} < \frac12.
$$
We conclude that there exists a nondegenerate interval $J$ contained
in $Y$ so that the winding number of $f_k(S_r)$ around each point of
$J$ is nonzero. By a degree theory argument similar to those used
previously, we conclude that $(0,y) \in f_k(B_r)$ for all $(0,y) \in
J$. This completes the proof of the claim. Since $f_k$ is Lipschitz,
the validity of this claim violates Lemma~\ref{GrushinMeasureZero}.
Hence the approximating sequence $f_k$ cannot exist and the proof of
Theorem~\ref{GrushinLipschitzDensity} is
complete.
\end{proof}

\begin{remark}
The proof of Theorem~\ref{GrushinLipschitzDensity} relied
on the construction of a bi-Lipschitz embedding of
$\Sph^1$ into $G$ whose image enclosed a nontrivial segment on
the $y$-axis $Y$. Such a bi-Lipschitz embedding admits no Lipschitz extension
to $\B^2$. By way of contrast, we have the following

\begin{theorem}\label{finalGrushin}
Let $\vi:\Sph^1\to G$ be an $L$-Lipschitz map with $\vi(\Sph^1)\cap Y = \emptyset$.
Then $\vi$ admits an $L$-Lipschitz extension $\tilde\vi:\B^2\to G$.
\end{theorem}
\end{remark}

Theorem~\ref{finalGrushin} follows from a generalized form of
Kirszbraun's theorem proved by Lang and Schroeder \cite[Theorem
A]{LangSchroeder}. For simplicity we only state a special case of the
Lang--Schroeder theorem sufficient for our purposes. See also
\cite[Theorem 1.3]{LangPavlovicSchroeder} for a related result (which
implies a weaker form of Theorem~\ref{finalGrushin} where the
Lipschitz constant of the extension may be larger than $L$).

\begin{theorem}[Lang--Schroeder]
Let $M$ and $N$ be Riemannian manifolds (possibly with boundary). Assume
that $M$ has all sectional curvatures bounded below by some $k\in\R$, while $N$ is
complete and has all sectional curvatures bounded above by the same $k$.
Then every $1$-Lipschitz map $f:A\to N$, $A\subset M$, admits a $1$-Lipschitz
extension $\tilde{f}:M\to N$.
\end{theorem}

\begin{proof}[Proof of Theorem~\ref{finalGrushin}]
Let $\vi:\Sph^1\to G$ be $L$-Lipschitz with $\vi(\Sph^1)\cap Y = \emptyset$.
By rescaling the metric in $\R^2$ if necessary, we may assume without loss of
generality that $L=1$.

Observe that $\vi(\Sph^1)$ is contained in $\{(x,y) \in G : |x|\ge\eps\}$ for
some $\eps>0$. In particular, $\vi(\Sph)$ lies completely within a Riemannian component
of $G$. We next compute the curvature of such a component. The Riemannian metric
in $\{(x,y):x>0\}$ is given by the length element
$ds^2=dx^2+x^{-2}dy^2$. Using Brioschi's formula
$$
K = \frac1{2\sqrt{EG}}
\left[ \left( \frac{G_x}{\sqrt{EG}} \right)_x +
\left( \frac{E_y}{\sqrt{EG}} \right)_y \right]
$$
for the (Gauss) curvature of a Riemannian metric $ds^2=Edx^2+Gdy^2$
with $E=1$ and $G=x^{-2}$, we obtain $K=-2x^{-2}$. Thus $K\le 
-2\eps^{-2}<0$ on $\{(x,y) \in G : x\ge\eps \}$. A similar computation
applies in $\{(x,y) \in G: x\le-\eps\}$.

It remains to show that the set $\{(x,y) \in G : x\ge\eps \}$, equipped with the
aforementioned Riemannian metric, is complete and geodesic. The latter assertion follows from
the explicit form \eqref{GGF} for the geodesics in $G$.
Since $\{(x,y) \in G : x\ge\eps \}$ is also locally compact, it is proper and
hence complete. The proof of Theorem~\ref{finalGrushin} is finished.
\end{proof}

\begin{question}
Does every Lipschitz map $\vi:\Sph^1\to G$ whose image encloses no segment
on the $y$-axis extend to a Lipschitz map of $\B^2$ into $G$?
\end{question}

\end{document}